\documentclass{amsart}
\pdfoutput=1
\usepackage[utf8]{inputenc}
\usepackage[T1]{fontenc}

\usepackage{mathrsfs}
\usepackage{mathtools}
\usepackage{thmtools}

\usepackage[margin=1.5in]{geometry}
\usepackage{tikz,tikz-3dplot}
\usepackage{enumitem}

\usepackage{todonotes}

\usepackage[
backend=biber,
style=numeric,
sorting=ynt
]{biblatex}

\usepackage[%
	colorlinks,%
	bookmarksdepth=3,%
	citecolor=blue%
]{hyperref}

\usepackage[capitalise]{cleveref}

\theoremstyle{plain}
\newtheorem{thm}{Theorem}[section]

\newtheorem{prop}[thm]{Proposition}
\newtheorem{cor}[thm]{Corollary}
\newtheorem{lemma}[thm]{Lemma}

\theoremstyle{definition}
\newtheorem{algorithm}[thm]{Algorithm}

\newtheorem{example}[thm]{Example}
\newtheorem{rem}[thm]{Remark}

\newtheorem{defin}[thm]{Definition}

\newcommand{\M}{\mathcal{M}}
\newcommand{\N}{\mathcal{N}}
\newcommand{\T}{\mathcal{T}}
\newcommand{\V}{\mathcal{V}}

\newcommand{\QQ}{\mathbb{Q}}
\newcommand{\CC}{\mathbb{C}}

\newcommand{\PP}{\mathbb{P}}
\newcommand{\ZZ}{\mathbb{Z}}
\newcommand{\RR}{\mathbb{R}}

\newcommand{\supp}{\operatorname{supp}}
\newcommand{\divv}{\operatorname{div}}
\newcommand{\rk}{\operatorname{rk}}

\newcommand{\D}{\mathcal{D}}

\newcommand{\Hom}{\operatorname{Hom}}
\newcommand{\cone}{\operatorname{cone}}

\newcommand{\GL}{\mathrm{GL}}
\newcommand{\SL}{\mathrm{SL}}
\newcommand{\Sp}{\mathrm{Sp}}

\begin{document}

\title{Toricness and Smoothness Criteria for Spherical Varieties}

\author[G.~Gagliardi]{Giuliano Gagliardi}
\address[G.~Gagliardi]{d-fine AG\\Brandschenkestrasse 150\\8002 Zurich\\Switzerland}
\email{}
\author[J.~Hofscheier]{Johannes Hofscheier}
\author[H.~Pearson]{Heath Pearson}
\address[J.~Hofscheier and H.~Pearson]{School of Mathematical Sciences\\University of Nottingham\\ Nottingham\\NG7 2RD\\UK}
\email{\{johannes.hofscheier, heath.pearson\}@nottingham.ac.uk}


\begin{abstract}
	We prove equivalent numerical conditions for a complete spherical variety to admit a toric structure, and for the smoothness of an arbitrary spherical variety along any given \(G\)-orbit. The conditions are in terms of spherical skeletons, a coarse ``subset'' of the Luna-Vust data of a spherical variety.
    Our smoothness criterion improves upon classical criteria by removing the dependency on external reference tables.
\end{abstract}

\maketitle{}

%--------------------------------------------------------------------------------
\section{Introduction}\noindent\phantomsection\label{sec:intro}In this paper we work over the field of complex numbers.

A \emph{spherical variety} is a normal irreducible variety \(X\) admitting the algebraic action of a connected reductive group \(G\), such that a Borel subgroup \(B\subseteq G\) has an open orbit in \(X\). Spherical varieties naturally extend the well-known classes of toric, and generalised flag varieties, which are combinatorially characterised up to equivariant isomorphism by lattice fans, and marked Dynkin diagrams, respectively. The full class of spherical varieties exhibit rich geometrical properties, and admit a hybrid description consisting of both fans and representation theoretical combinatorics.

The results of this paper are equivalent conditions of two geometrical properties in terms of \emph{spherical skeletons}—a simpler coarsening of the full spherical combinatorial data which determines the un-graded Cox ring of a spherical variety, analogously to the root datum of a reductive group encoding the core of its geometry, yet not capturing its isomorphism class. We show equivalent combinatorial conditions for:

\begin{enumerate}
    \item a complete spherical variety to admit a toric structure (\Cref{main thm}),\label{one}
    \item a spherical variety to be smooth along any given \(G\)-orbit (\Cref{smoothness theorem}).\label{two}
\end{enumerate}

These results are corollaries of the main theorem of this paper (\Cref{main thm}), which proves~\cite[Conjecture~1.4]{GH15}. 

\subsection{Main theorem}

Recall that spherical varieties are classified by \emph{Luna-Vust data},  which is introduced in \Cref{sec:background}, and used freely here. The spherical skeleton is a ``subset'' of the Luna-Vust datum, which was introduced in \cite{GH15} to approach the generalised Mukai conjecture for spherical Fano varieties. This conjecture was recently proven by the authors in \cite{GHPsphericalMukai} and forms the starting point of this work. 

Let \(R\) be the root system of a reductive group \(G\) with respect to a fixed maximal torus \(T \subseteq G\), with simple roots \(S\subseteq R\).

\begin{defin}[{\cite[Definition~5.1]{GH15}}]\label{def:spherical skeleton}
	The \emph{spherical skeleton} \(\mathscr{R}_X\) of a spherical variety \(X\) is a quadruple, equipped with two maps
    \[
		\mathscr{R}_X\coloneq(\Sigma,S^p,\D^a,\Gamma), \quad \varsigma \colon \D^a \to \mathcal{P}(S), \quad \rho' \colon \mathcal{D}^a \cup \Gamma \to \Lambda^*\coloneqq\Hom(\Lambda,\ZZ).
    \]
	Here, \(\Lambda\) is the sublattice of \(\M\) spanned by the spherical roots \(\Sigma\) of \(X\); \(S^p\) is the subset of simple roots whose minimal parabolic subgroups stabilise every color of \(X\); \(\D^a\) is the set of type \(a\) colors of \(X\); and \(\Gamma\) is the set of \(G\)-invariant prime divisors of \(X\).
    For each color \(D\in\D\) (of any type) of \(X\), \(\varsigma(D)\) is the set of simple roots whose associated minimal parabolic subgroup do not fix \(D\). To each \(B\)-invariant divisor \(D\in\D\cup\Gamma\), \(\rho'\) associates the dual vector \(\rho'(D)\coloneqq\rho(D)|_{\Lambda^*}\).
\end{defin}

The anticanonical class of a spherical variety \(X\) has a distinguished representative, namely
\[
    -K_X=\sum_{D\in\Delta}m_D D,
\]
where the coefficients \(m_D\) can be obtained from the spherical skeleton (see \Cref{rem:coeffs-anticanonical}) and \(\Delta \coloneq \D \cup \Gamma\) denotes the set of \(B\)-invariant prime divisors of \(X\) (which can be extracted from the spherical skeleton \(\mathscr{R}_X\)).

To each spherical skeleton \(\mathscr{R}_X\), we associate a polyhedron
\[
	\mathcal{Q}_{\mathscr{R}_X}^* \coloneq \bigcap_{D\in\Delta} \left\{v\in\Lambda_\QQ^* \, | \, \langle \rho'(D), v \rangle \geq -m_D \right\} \subseteq \Lambda_\QQ.
\]

Furthermore, the spherical skeleton comes with the following combinatorial function
\[
	\widetilde{\wp}(\mathscr{R}_X) \coloneqq |R^+| - |R_{S^p}^+| -\sup\left\{\sum_{D\in\Delta}m_D-1+\langle\rho'(D), \vartheta\rangle \, \mid \, \vartheta \in \mathcal{Q}^*_{\mathscr{R}_X}\cap\mathrm{cone}(\Sigma)\right\},
\]
where \(R_{S^p}\) is the sub-root system generated by \(S^p\), and \(R^+\) (resp.~\(R_{S^p}^+\)) is the set of positive roots of \(R\) (resp.~\(R_{S^p}\)).
\begin{rem}
    This definition is different to the definition of \(\wp(\mathscr{R}_X)\) in~\cite{GH15}, and they are related in the following way: \(\widetilde{\wp}(\mathscr{R}_X)=|R^+|-|R_{S^p}^+|-\wp(\mathscr{R}_X)\). Moreover either \(\widetilde{\wp}(\mathscr{R}_X)\le|R^+|-|R_{S^p}^+|\) is a rational number, or it is negative infinity, by \cite[Proposition~5.6]{GH15}.
\end{rem}

We now state our main theorem, which proves \cite[Conjecture~1.4]{GH15}, and answers \eqref{one}.

\begin{thm}\label{main thm} Let \(X\) be a complete spherical variety, then
\[
\widetilde{\wp}(\mathscr{R}_X)\ge0,
\]
with equality if and only if \(X\) is isomorphic to a toric variety.
\end{thm}

We now introduce the smoothness criterion which answers \eqref{two}. It utilises the following operation on spherical skeletons.

\begin{defin}[{\cite[Definition~7.1]{GH15}}]
	Given a spherical skeleton \(\mathscr{R}=(\Sigma,S^p,\D^a,\Gamma)\) and a subset \(I\subseteq\Delta\), define another spherical skeleton as follows.
	Let \(S_I\subseteq S\) be the subset of simple roots whose minimal parabolic subgroups do not fix some divisor in \(I\).
	Let \(R_I\) be the root system generated by \(S_I\).
	Now define
	\begin{gather*}
		\Sigma_I\coloneq\Sigma\cap\mathrm{span}_\ZZ S_I,\\
		S_I^p\coloneq S^p\cap S_I,\\
		\D_I^a\coloneq\{\text{type \(a\) colors with spherical roots in } \Sigma_I\},\\
		\Gamma_I\coloneq \{G\text{-invariant divisors in } I\}.
	\end{gather*}
	The \emph{localisation of \(\mathscr{R}\) at \(I\)} is defined as the spherical skeleton \(\mathscr{R}_I\coloneq(\Sigma_I,S^p_I,\D^a_I,\Gamma_I)\), equipped with the maps \(\rho_I'\coloneq\rho'|_{\Lambda_I^*}\) and \(\varsigma_I\coloneq\varsigma|_{{\D}_I}\), with \(\Lambda_I\coloneq\mathrm{span}_\ZZ(\Sigma_I)\).

\end{defin}

\begin{thm}\noindent\phantomsection\label{smoothness theorem}
	Let \(X\) be a spherical variety, and \(Y\subseteq X\) a \(G\)-orbit.
	Let \(I\subseteq\Delta\) be the set of \(B\)-invariant divisors \(D\) with \(Y\subseteq\overline{D}\).
	Then \(X\) is smooth along \(Y\) if and only if \(\widetilde{\wp}(\mathscr{R}_I)<1\).
\end{thm}

The smoothness criterion, \Cref{smoothness theorem}, follows from \Cref{main thm} and \cite[Theorem~7.2]{GH15}. It is based on the smoothness criterion of Camus from 2001 for type \(A\) spherical varieties \cite{Camus2001}, which was generalised by the second author to all types in 2014 \cite{GagliardiMultiplicityFree}, and conjecturally reformulated in terms of spherical skeletons in 2015 by the first two authors \cite{GH15}.

We note other results in the direction of smoothness criteria for spherical varieties. The first smoothness criterion for spherical varieties was obtained in 1991 by Brion \cite{Brion1991}.
A combinatorial smoothness criterion for horospherical varieties was obtained independently by Pasquier in 2006 \cite{Pasquier2006} and Timashev \cite{Timashev} in 2011. Alternatively in 2013, Batyrev and Moreau \cite{BatyrevMoreau2013} proved a smoothness criterion for horospherical varieties using stringy Euler numbers, and this criterion is conjectured to extend to arbitrary spherical varieties. The smooth affine spherical varieties were classified by Knop and Van Steirteghem in 2006 \cite{KnopVanSteirteghem2006}. In 2017, the weight monoids of smooth affine spherical varieties were combinatorially classified by Pezzini and Van Steirteghem \cite{PezziniVanSteirteghem2017}.

\subsection{Technique of the proof} To each spherical variety, we may associate the quantity \(\widetilde{\wp}(X)\coloneq\widetilde{\wp}(\mathscr{R}_X)\). The starting point of this paper is the following result of the authors, which was used to prove the generalised Mukai conjecture for \(\QQ\)-factorial spherical Fano varieties in \cite{GHPsphericalMukai}.

\begin{thm}[{\cite[Corollary of Theorem~1.6]{GHPsphericalMukai}}]
    \noindent\phantomsection\label{Spherical Mukai} Let \(X\) be a complete \(\QQ\)-Gorenstein spherical variety, then
    \[
        \widetilde{\wp}(X)\ge0,
    \]
    with \(\widetilde{\wp}(X)<1\) only if \(X\) is isomorphic to a toric variety.
\end{thm}

The main theorem of this paper lifts the \(\QQ\)-Gorenstein hypothesis of \Cref{Spherical Mukai}, and proves \(\widetilde{\wp}(X)=0\) when \(X\) is isomorphic to a toric variety. In both of these steps we frequently use the flexibility enabled by the coarseness of spherical skeletons. Together, this proves \Cref{main thm}.

Let us now give an overview of the structure of the paper.
Firstly, in \Cref{sec:complete case proof} we introduce an algorithm which replaces a complete spherical variety \(X\) with a complete \(\QQ\)-Gorenstein spherical variety \(X'\) such that \(\widetilde{\wp}(X)=\widetilde{\wp}(X')\). In \Cref{sec:proof-of-gorenstein-algorithm} we prove the correctness of this algorithm. Next, in \Cref{sec:multiplicity free} we show the statement: ``\(\widetilde{\wp}(X)=0\) if \(X\) is isomorphic to a toric variety'', is equivalent to \(\widetilde{\wp}(V)=0\) where \(V\) is a member of one of \(42\) families of multiplicity free spaces. Finally, in \Cref{Table} we explicitly verify this equality for each of these \(42\) families, completing the proof of \Cref{main thm}.

In \Cref{sec:examples} we give examples illustrating the smoothness criterion \Cref{smoothness theorem}.

%--------------------------------------------------------------------------------

\section{Defining the \texorpdfstring{\(\widetilde{\wp}\)}{p-tilde} function}\noindent\phantomsection\label{sec:background}In this section we briefly overview aspects of the combinatorial description of spherical varieties which are required to define \(\widetilde{\wp}(\cdot)\). A general reference is~\cite{Timashev}. 

In this paper, \(G\) denotes a connected reductive algebraic group, \(B\subseteq G\) a Borel subgroup, \(X\) a spherical \(G\)-variety, and \(H\) the stabiliser of a point in the open \(B\)-orbit of \(X\).

Every spherical variety \(X\) is a \(G\)-equivariant open embedding of a \(G\)-homogeneous space \(G/H\) with an open \(B\)-orbit, into a normal irreducible \(G\)-variety. Such \(G/H\) are called \emph{spherical homogeneous spaces}, and such embeddings \(G/H\hookrightarrow X\) are called \emph{spherical embeddings}. These form the two natural halves of the combinatorial classification of \(G\)-spherical varieties for a fixed \(G\):
\begin{enumerate}
\item{} The classification of spherical homogeneous spaces \(G/H\)—classified by \emph{Luna data}, also called \emph{spherical homogeneous data} (see~\cite[Section~30.11]{Timashev} or~\cite{FoundationsLunaData}).
\item{} The classification of spherical embeddings \(G/H\hookrightarrow X\) of a fixed spherical homogeneous space \(G/H\)—classified by the \emph{Luna-Vust theory} of \emph{colored fans} (see~\cite{Knop2012}).
\end{enumerate}

We introduce the Luna-Vust theory in this section, and in \Cref{sec:complete case proof}, we introduce the required elements of the Luna data as they are needed.

Firstly, we associate to \(X\) its character lattice of \(B\)-weights
\[
\M\coloneq\{\chi\mid\text{there exists } f_\chi\in\CC(X)\text{ with }b\cdot f_\chi(x)=\chi(b)f_\chi(x) \text{ for all }x\in X,\,b\in B\}\subseteq\mathfrak{X}(B).
\]

We denote by \(\rk X\) the rank of the lattice \(\M\).

Write \(\Delta\coloneq\Delta_X\) to denote the finite set of \(B\)-invariant prime divisors in \(X\). As in the classification of toric varieties by fans, to each \(D \in \Delta\) there is a dual vector \(\rho(D) \in \N \coloneqq \Hom(\M,\ZZ)\) defined by \(\langle \rho(D), \chi\rangle \coloneqq \nu_D(f_\chi)\), where \(\langle \cdot, \cdot\rangle \colon \N \times \M \to \ZZ\) is the dual pairing. 

The dual vectors associated to \(G\)-invariant prime divisors of \(X\) are primitive vectors in \(\N\subset\N_\QQ\coloneq\N\otimes_\ZZ\QQ\) which lie within a certain full-dimensional rational polyhedral cone, the \emph{valuation cone}:
\[
	\V \coloneqq \cone\left\{ \rho(D)\mid D\in\Delta_X \text{ is }G\text{-invariant, for a spherical embedding }G/H\hookrightarrow X\right\}\subseteq\N_\QQ.
\]

Let \(G/H \hookrightarrow X\) be a spherical embedding. 
There is a finite set \(\D\subseteq\Delta\) of \(B\)-invariant prime divisors in \(X\), called \emph{colors}, arising as the closures of \(B\)-invariant divisors in the spherical homogeneous space \(G/H\).
A \emph{colored cone} is a pair \((\mathcal{C}, \mathcal{F})\) such that \(\mathcal{F} \subseteq \mathcal{D}\) and \(\mathcal{C} \subseteq \N_\QQ\) is a strictly convex cone generated by \(\rho(\mathcal{F})\) and finitely many elements of \(\V \cap \N\), satisfying \(\rho(D)\neq0\) for all \(D\in\mathcal{F}\), and \(\mathcal{C}^\circ \cap \V \neq \emptyset\), where \(\mathcal{C}^\circ\) denotes the relative interior of \(\mathcal{C}\).
Furthermore, a face of a colored cone \((\mathcal{C},\mathcal{F})\) is a colored cone \((\mathcal{C}', \mathcal{F}')\) such that \(\mathcal{C}'\) is a face of \(\mathcal{C}\) in the usual convex geometric sense, and \(\mathcal{F}' = \{ D \in \mathcal{F} \mid \rho(D) \in \mathcal{C}'\}\).

With this notion, one defines a \emph{colored fan} as a finite collection \(\mathbb{F}=\{(\mathcal{C}_i, \mathcal{F}_i)\}\) of colored cones, such that each face of a colored cone in \(\mathbb{F}\) is also in \(\mathbb{F}\), and such that each point in the valuation cone is contained in the relative interior of at most one colored cone in \(\mathbb{F}\).

Then, according to the Luna-Vust theory, there is a bijection between colored fans and isomorphism classes of spherical embeddings of \(G/H\).
Moreover, under this bijection, complete spherical varieties correspond to colored fans \(\mathbb{F}\) which cover the valuation cone, i.e.\ \(\V \subseteq |\mathbb{F}|\) where \(|\mathbb{F}|\coloneq \bigcup_i \mathcal{C}_i\subseteq \N_\QQ\) denotes the \emph{support} of the colored fan \(\mathbb{F}\).

In~\cite[Proposition~4.1]{BrionAnticanonical}, Brion obtained a natural representative of the anticanonical class as a positive linear combination of the \(B\)-invariant divisors
\[
	-K_X \coloneqq \sum_{D \in \Delta} m_D D.
\]
The coefficients \(m_D\ge1\) may be computed using the Luna data (see \Cref{rem:coeffs-anticanonical}).

Next, we associate to \(X\) a polyhedron:
\[
	Q^*\coloneqq \bigcap_{D \in \Delta} \left\{ v \in \M_{\QQ} \mid \langle \rho(D), v \rangle \ge -m_D \right\}\subseteq\M_\QQ.
\]
When \(X\) is complete, \(Q^*\) is a polytope. Moreover if \(X\) is \(\QQ\)-Gorenstein Fano, then \(Q^*\) is a translation of the moment polytope of the anticanonical divisor of \(X\).

\begin{defin}\noindent\phantomsection\label{def: P}
	Let \(X\) be a spherical variety, then 
	\[
		\widetilde{\wp}(X) \coloneq \dim X-\rk X-\sup \left\{ \sum_{D\in\Delta}\left(m_D-1+\big\langle\rho(D),\vartheta\big\rangle\right)\,\big|\,\vartheta\in Q^*\cap\T \right\},
	\]
	where \(\T \coloneq -\V^\vee\) is the negative of the dual cone of \(\V\).
\end{defin}

\begin{rem}
This definition relates to the definition of \(\wp(X)\) in~\cite{GH15} in the following way: \(\widetilde{\wp}(X)=\dim X-\rk X-\wp(X)\).
\end{rem}

%--------------------------------------------------------------------------------
\section{The inequality in \texorpdfstring{\Cref{main thm}}{Theorem~\ref{main thm}}}\noindent\phantomsection\label{sec:complete case proof}In this section we prove that the \(\QQ\)-Gorenstein assumption of \Cref{Spherical Mukai} may be lifted.
This is achieved by reducing the complete case to the complete \(\QQ\)-Gorenstein setting of \Cref{Spherical Mukai} by operating on the spherical combinatorial data.

In particular, given the colored fan of a complete spherical variety \(X\), in \Cref{Gorenstein algorithm} we construct the colored fan of a complete \(\QQ\)-Gorenstein spherical variety \(X'\) in a way such that \(\widetilde{\wp}(X)=\widetilde{\wp}(X')\) is preserved, and such that a spherical variety isomorphic to a toric variety is sent to another spherical variety isomorphic to a toric variety. 

Throughout, we describe spherical homogeneous spaces combinatorially using their \emph{Luna data} (also called \emph{spherical homogeneous data}); for details we refer to~\cite[Section~30.11]{Timashev} or~\cite{FoundationsLunaData}.

To begin, we make some observations about the dual vectors \(\rho(D)\in\N\) associated to the colors \(D\in\D\) of a spherical embedding.

Let \(X\) be a \(G\)-spherical variety.
Denote the root system of \(G\) with respect to a fixed maximal torus \(T \subseteq G\) by \(R\), and let \(S\subseteq R\) be the choice of simple roots corresponding to our choice of Borel subgroup \(B \subseteq G\).
We say that a color \(D\in\D\) is \emph{moved} by \(\alpha\in S\) if the minimal parabolic subgroup \(P_\alpha\), that contains \(B\) and is associated to \(\alpha\), satisfies \(P_\alpha\cdot D\neq D\).
Denote by \(S^p\) the set of simple roots which move no colors of \(X\).
Recall that the set of colors of a spherical variety is split into three disjoint cases:
\[
    \text{a color } D \in \D \text{ is said to be of type}\begin{cases}
        a & \text{if \(D\) is moved by some } \alpha \in S \cap \Sigma,\\
        2a & \text{if \(D\) is moved by some } \alpha \in S \cap \frac12\Sigma,\\
        b & \text{otherwise}.
    \end{cases}
\]
\begin{rem}
    We list some properties of colors which are required below and follow directly from the axioms of Luna data (see, for instance, \cite[Definition~30.21]{Timashev}):
    \begin{itemize}
    \item
        Colors of type \(a\) come in pairs, that is, for each \(\alpha \in S \cap \Sigma\) there are two colors \(D_\alpha^+\) and \(D_\alpha^-\) that are moved by \(\alpha\) such that \(\rho(D_\alpha^+) + \rho(D_\alpha^-)=\alpha^\vee|_\M\).
    \item 
        We have \(\langle \rho(D_\alpha^\pm), \gamma \rangle \le 1\) for all \(\gamma \in \Sigma\), with equality if and only if \(\gamma = \alpha\).
    \item 
        Each spherical root is a non-negative linear combination of simple roots.
        We write \(\supp{\gamma}\subseteq S\) for the set of simple roots that appear in \(\gamma\).
    \end{itemize}
\end{rem}

\begin{defin}
    Call a pair of colors \(D_1,D_2\) \emph{codirectional} if \(\rho(D_1)=k\rho(D_2)\) for some rational number \(k>0\). 
\end{defin}

\begin{prop}\noindent\phantomsection\label{colors codirectional}
    Suppose that \(D_1,D_2\in\D\) are codirectional.
    Then either: \(D_1,D_2\) are a pair of colors of type \(a\) with \(\rho(D_1)=\rho(D_2)\), or \(D_1,D_2\) are colors of type \(b\).
\end{prop}

\Cref{colors codirectional} is a corollary of the following two propositions.
    
\begin{prop}\noindent\phantomsection\label{different-spherical-roots-non-codirectional}
    Let \(D\) be a color of type \(a\) or \(2a\) and let \(D'\neq D\) be a color that is not of type \(a\).
    Then \(D\) and \(D'\) are not codirectional. 
\end{prop}

\begin{proof}
    Let \(\gamma = \lambda \alpha \in \Sigma\) (\(\lambda \in \{1,2\}\)) be the spherical root associated to \(D\).
    Notice that \(\langle \rho(D), \gamma \rangle > 0\).
    Assume for a contradiction that \(D'\neq D\) is a color not of type \(a\) which is codirectional with \(D\), so \(\rho(D')=\ell\beta^\vee|_\M\) for some \(\ell>0\) and \(\alpha\neq\beta\in S\), and \(\rho(D)=k\rho(D')\) for \(k>0\).
    Since \(\beta\notin\supp{\gamma}=\{\alpha\}\),  it follows that \(0<\langle\rho(D),\gamma\rangle=\langle k\rho(D'),\gamma\rangle=k\ell\langle\beta^\vee,\gamma\rangle \leq 0\), so no such \(k\) exists.
\end{proof}

\begin{prop}\noindent\phantomsection\label{prop:two-codirectional-cols-type-a}
    Let \(D_1, D_2\) be two codirectional colors of type \(a\).
    Then \(D_1\) and \(D_2\) are moved by the same simple root \(\alpha \in S \cap \Sigma\) and \(\rho(D_1) = \rho(D_2)\).
\end{prop}

\begin{proof}
    Suppose that \(D_1\) is moved by \(\alpha \in S \cap \Sigma\), \(D_2\) is moved by \(\beta \in S \cap \Sigma\), and \(\rho(D_2)=k\rho(D_1)\) for some \(k>0\).
    Let \(D_\alpha^+,D_\alpha^-\) (resp.~\(D_\beta^+, D_\beta^-\)) be the pair of colors of type \(a\) that are moved by \(\alpha\) (resp.~\(\beta\)).
    Then \(D_1 \in \{D_\alpha^-, D_\alpha^+\}\) and \(D_2 \in \{D_\beta^-, D_\beta^+\}\).
    Without loss of generality assume \(D_2 = D_\beta^+\).
    Then, since \(\langle \rho(D_1), \alpha \rangle = 1\), we have \(\langle \rho(D_\beta^+), \alpha\rangle = k\).
    The colors \(D_\beta^\pm\) satisfy \(\langle \rho(D_\beta^+),\alpha\rangle + \langle\rho(D_\beta^-),\alpha\rangle = \langle\beta^\vee,\alpha\rangle\), so \(k=\langle\beta^\vee,\alpha\rangle-\langle\rho(D_\beta^-),\alpha\rangle\in\ZZ\).
    Then as \(0<k=\langle \rho(D_\beta^+),\alpha\rangle\le1\), the only possibility is \(k=1\).
    In particular, we get \(\alpha = \beta\).
\end{proof}

\begin{proof}[Proof of \Cref{colors codirectional}]
    If \(D_1\) and \(D_2\) are not both of type \(b\), then the unordered tuple of types of \(D_1\) and \(D_2\), denoted \((t_1,t_2)\), are one of the following cases:
    \[
        (t_1,t_2) \in \{(a,a), (a,2a), (a,b), (2a,2a), (2a,b)\}.
    \]
    By \Cref{different-spherical-roots-non-codirectional}, \(D_1\) and \(D_2\) are not codirectional in each of these cases, except when \(D_1\) and \(D_2\) are both of type \(a\), in which case \Cref{prop:two-codirectional-cols-type-a} shows that \(\rho(D_1)=\rho(D_2)\).
    This completes the proof of \Cref{colors codirectional}.
\end{proof}

\begin{rem}\noindent\phantomsection\label{rem:coeffs-anticanonical}
	We list some properties of the coefficients \(m_D\) in the standard spherical anticanonical divisor of a spherical variety \(X\hookleftarrow G/H\).
	By~\cite[Theorem~4.2]{BrionAnticanonical} (see also~\cite[\S3.6]{LunaAnticanonical} or~\cite[Theorem~1.5]{GorensteinFano}), the anticanonical sheaf of \(G/H\) is associated to the \(B\)-weight \(\kappa=2(\rho_S-\rho_{S^p})\), where for \(I \subseteq S\) we denote by \(\rho_I\) the half sum of positive roots in the root system generated by \(I\). 
    For each color \(D \in \D\) which is moved by a simple root \(\alpha \in S\), the coefficients \(m_D\) are:
	\begin{align*}
		m_D &= \frac12 \langle \alpha^\vee, \kappa\rangle = 1 && \text{for \(D\) of type \(a\) or \(2a\)},\\
		m_D &= \langle \alpha^\vee, \kappa \rangle \ge 2 && \text{for \(D\) of type \(b\)}.
	\end{align*}
\end{rem}

\begin{prop}\noindent\phantomsection\label{Gorenstein Luna datum}
    If \(\ell\kappa\in\M\) for some \(\ell > 0\), then for every pair of colors \(D_1,D_2\), either \(\rho(D_1)\) and \(\rho(D_2)\) span different rays or \(\rho(D_1)/m_{D_1}=\rho(D_2)/m_{D_2}\).
\end{prop}

\begin{proof}
    By \Cref{colors codirectional}, there are two cases in which \(\rho(D_1)\) and \(\rho(D_2)\) span the same ray.
    If these colors are both of type a, then \(m_{D_1}=m_{D_2}=1\), and thus, by \Cref{colors codirectional}, \(\rho(D_1)=\rho(D_2)\).
    If both these colors are of type b, and \(\rho(D_1)=k\rho(D_2)\) for \(k>0\), then \(\ell m_{D_1} = \langle \rho(D_1), \ell \kappa \rangle = \langle k \rho(D_2), \ell \kappa \rangle = k \ell m_{D_2}\).
    In particular, \(k = m_{D_1}/m_{D_2}\), from which the statement follows.
\end{proof}

\begin{prop}\noindent\phantomsection\label{prop-add-kappa}
    Given a Luna datum with \(\ell' \kappa \not\in \M\) for all \(\ell'\in\ZZ_{\ge0}\), define \(\M'\coloneq\M \oplus \ZZ\kappa \subseteq \mathfrak{X}(B)\).
	Applying both the following modifications yields a genuine Luna datum
    \begin{itemize}
    \item
        replace \(\M\) by \(\M'\), so that \(\N' = \Hom(\M', \ZZ) \cong \N \oplus \ZZ\),
    \item
        replace \(\rho \colon \D \to \N\) by \(\rho'\colon \D \to \N' = \N \oplus \ZZ; D \mapsto (\rho(D), m_D)\).
    \end{itemize}
\end{prop}

\begin{proof}
    We verify the axioms of a Luna datum, using the numbering of axioms given in~\cite[Definition~30.21]{Timashev}.

    Clearly \(\Sigma \subseteq \M \subseteq \M'\).
    Firstly, we require every spherical root \(\gamma\in\Sigma\) to be primitive in \(\M'\). Indeed as \(\kappa\notin\M\), we have \(\gamma=\gamma+0\kappa\in\M\oplus\ZZ\kappa=\M'\), so each spherical root remains primitive in \(\M'\).

	Notice that for all \(\gamma \in \Sigma\), we have \(\gamma \in \M \subseteq \M'\), and thus \(\langle \rho'(D), \gamma\rangle = \langle \rho(D), \gamma \rangle\) for each color \(D \in \D\) of type \(a\).
	Therefore, axiom (A1) for the modified Luna datum follows from the original Luna datum.

	Let \(\alpha \in \Sigma \cap S\) and let \(D_\alpha^+,D_\alpha^-\) be the pair of colors of type \(a\) that are moved by \(\alpha\).
	Then we have
	\[
		\rho'(D_\alpha^+) + \rho'(D_\alpha^-) = (\rho(D_\alpha^+),1) + (\rho(D_\alpha^-),1) = (\rho(D_\alpha^+)+\rho(D_\alpha^-),2) = ( \alpha^\vee|_\M, \langle\alpha^\vee, \kappa\rangle),
	\]
	where for the last equality we have used \Cref{rem:coeffs-anticanonical}, i.e.\ \(\langle \alpha^\vee, \kappa \rangle = 2\).
	This shows \(\rho'(D_\alpha^+)+\rho'(D_\alpha^-) = \alpha^\vee|_{\M'}\), and thus verifies axiom (A2) in~\cite[Definition~30.21]{Timashev}.

	Axiom~(A3) in~\cite[Definition~30.21]{Timashev} is trivially satisfied.

	Let \(\alpha \in S \cap \frac12\Sigma\).
	By \Cref{rem:coeffs-anticanonical}, we have that \(\langle \alpha^\vee,\kappa\rangle =2\), and thus \(\langle \alpha^\vee, \M'\rangle \subseteq 2\ZZ\).
	Since \(\Sigma \subseteq \M\), it follows from axiom (\(\Sigma1\)) for the original Luna datum that \(\langle \alpha^\vee, \Sigma \setminus \{2\alpha\}\rangle \le 0\).
	Therefore axiom (\(\Sigma1\)) holds for the modified Luna datum.

	Suppose \(\alpha, \beta \in S\), \(\alpha \perp \beta\), and \(\gamma \coloneqq \alpha+\beta \in \Sigma \cup 2\Sigma\).
	Then the spherical root \(\gamma\) comes with a color \(D \in \D\) whose coefficient in the anticanonical class is given by \(\langle \alpha^\vee, \kappa\rangle = \langle \beta^\vee, \kappa \rangle = m_D\).
	Hence, by axiom (\(\Sigma2\)) for the original Luna datum, it follows that \(\alpha^\vee = \beta^\vee\) on \(\M'\).

	For the final axiom (S), it remains to show \(\langle \alpha^\vee, \kappa \rangle = 0\) for each \(\alpha \in S^p\).
	This follows from the following general observation: let \(I \subseteq S\) be a subset and \(\rho_I\) the half sum of positive roots in the root system generated by \(I\). It is well known that \(\rho_I\) is equal to the sum of fundamental weights in the root system generated by \(I\).
	Thus, for each \(\alpha \in I\), we have \(\langle \alpha^\vee, \rho_I\rangle = 1\).
	Returning to our original question: recall \(\kappa = 2(\rho_S-\rho_{S^p})\).
	Then for each \(\alpha \in S^p\), \(\langle\alpha^\vee,\kappa\rangle=2(\langle\alpha^\vee,\rho_{S}\rangle-\langle\alpha^\vee,\rho_{S^p}\rangle)=0\), as desired.
\end{proof}

\begin{rem}\noindent\phantomsection\label{new-valuation-cone}
	Notice that the valuation cone \(\V'\) of the modified Luna datum in \Cref{prop-add-kappa} satisfies \(\V' = \V \oplus \QQ \subseteq \N'_\QQ = \N_\QQ \oplus \QQ\) where \(\V \subseteq \N_\QQ\) is the valuation cone of the original Luna datum.
\end{rem}

We now describe a class of colored fans corresponding to \(\QQ\)-Gorenstein spherical varieties (see \Cref{Q-Gorenstein fan}).
Then, we present an algorithm which converts a complete colored fan into a complete colored fan which is a member of this class.

\begin{rem}\noindent\phantomsection\label{Q-Gorenstein fan}
	Recall that a spherical variety is \(\QQ\)-Gorenstein if and only if each colored cone \((\mathcal{C},\mathcal{F})\) admits a linear function \(f_{(\mathcal{C},\mathcal{F})} \in \M_\QQ\) such that \(m_D = \langle \rho(D), f_{(\mathcal{C},\mathcal{F})}\rangle\) for each ray \(\QQ_{\ge0}\rho(D)\) of \(\mathcal{C}\) (with \(D \in \Delta\)), and \(m_{D'}=\langle \rho(D'), f_{(\mathcal{C},\mathcal{F})}\rangle\) for each \(D'\in\mathcal{F}\) (see, for instance, \cite[first Proposition in Section~3.1]{BrionBasepointFree}).

    In particular, consider a colored fan where every colored cone \((\mathcal{C},\mathcal{F})\) satisfies:
    \begin{enumerate}
    \item
        \(\mathcal{C}\) is a simplicial cone, i.e.\ the rays of \(\mathcal{C}\) are generated by \(\QQ\)-linearly independent vectors;
    \item
		for every color \(D\in\mathcal{F}\), \(\QQ_{\ge0}\rho(D)\) is a ray of \(\mathcal{C}\);
    \item
		for every pair of colors \(D_1,D_2\in \mathcal{F}\) with the same ray \(\QQ_{\geq0}\rho(D_1)=\QQ_{\geq0}\rho(D_2)\), demand that \(\rho(D_1)/m_{D_1}=\rho(D_2)/m_{D_2}\).
    \end{enumerate}
    Remark that the spherical variety of such a colored fan is \(\QQ\)-Gorenstein.
\end{rem}

\begin{rem}
    Recall from \Cref{sec:background} that each colored cone \((\mathcal{C}, \mathcal{F})\) in a colored fan consists of a strictly convex rational polyhedral cone \(\mathcal{C}\subseteq \N_\QQ\) and a subset \(\mathcal{F} \subseteq \D\) such that
    \[
		\mathcal{C}=\cone(\{v_1,\dots,v_n\} \cup \rho(\mathcal{F}))\subseteq\N_\QQ
    \]
	for some \(v_1, \dots, v_n\in\V\cap\N\) such that \(\rho(D)\neq0\) for \(D \in \mathcal{F}\), and the relative interior of \(\mathcal{C}\) intersects the valuation cone, i.e.\ \({\mathcal{C}}^\circ \cap \V \neq \emptyset\).
    If \((\mathcal{C}, \mathcal{F})\) satisfies these properties, but not necessarily the intersection property with the valuation cone, we call it an \emph{abstract colored cone}. We define \emph{abstract colored fans} as collections of abstract colored cones such that each abstract colored face of an abstract colored cone in the fan is also in the fan, and every point in the valuation cone lies in the relative interior of at most one abstract colored cone of the fan.
	Clearly, by removing those abstract colored cones from an abstract colored fan that are not genuine colored cones, one obtains a genuine colored fan.
	Furthermore, it is easy to show that when intersected with the valuation cone, the supports of both the abstract colored fan and its corresponding ``genuine'' colored fan coincide.

\end{rem}

\begin{prop}\noindent\phantomsection\label{reduction to gorenstein}
    For every complete spherical variety \(X\hookleftarrow G/H\), there exists a complete \(\QQ\)-Gorenstein spherical variety \(X'\hookleftarrow G/H'\) such that \(\widetilde{\wp}(X')=\widetilde{\wp}(X)\).

    Indeed, the Luna datum and colored fan of such an \(X'\) can be taken to be the output of \Cref{Gorenstein algorithm}.
\end{prop}

\begin{algorithm}\noindent\phantomsection\label{Gorenstein algorithm}
    The input of the algorithm is the colored fan \(\mathbb{F}\) and Luna datum of the complete spherical variety \(X\hookleftarrow G/H\).
    
    \begin{enumerate}
    \item\noindent\phantomsection\label{first-step}
    Firstly, if there is an \(\ell>0\) such that \(\ell\kappa\in\M\), write \(\mathbb{F}_3\coloneq\mathbb{F}\) and go to Step~\eqref{item:make-colors-rays}.
    
        Otherwise, replace the Luna datum of \(G/H\) with the new Luna datum specified in \Cref{prop-add-kappa}.
        Let the resulting spherical homogeneous space be \(G/H'\), and signify its associated spherical combinatorial data by adding an apostrophe to the corresponding notation for \(G/H\).
        Next, go to Step~\eqref{item:add-kappa}. 
    
    \item\noindent\phantomsection\label{item:add-kappa}
        Notice that by \Cref{new-valuation-cone}, the new Luna datum has \(\N'=\N\oplus\ZZ\), and colors \(\D' = \D\).
        Using this Luna datum, we define an embedding \(G/H' \hookrightarrow X'\) combinatorially via a colored fan \(\mathbb{F}'\), whose construction begins here as the colored fan \(\mathbb{F}_2\), and is completed in Step~\eqref{item:final-step} as \(\mathbb{F}'\coloneq\mathbb{F}_5\).
        Firstly, we describe the dual vectors \(\rho'(D) \in \N'\) for \(D\in\Delta\).
		The map \(\rho' \colon \Delta \to \N'\) satisfies \(\rho'(D)=(\rho(D),m_D)\in \N'\) for colors \(D \in \D'=\D\), and we set \(\rho'(X_i) \coloneqq (\rho(X_i),1)\) for \(G\)-invariant divisors \(X_i \in \Delta\).
		The colored cones of \(\mathbb{F}_2\) are as follows.
		Let \({(D_i)}_{i\in I}\) with \(D_i\in\Delta\) be the family of \(B\)-invariant divisors corresponding to the rays and colors of a cone \((\mathcal{C},\mathcal{F})\) in \(\mathbb{F}\).
		Then the colored cones of \(\mathbb{F}_2\) are:
        \[
            (\mathcal{C}', \mathcal{F}') \coloneqq (\cone(\{\rho'(D_i) \mid i \in I\}), \mathcal{F}) \subseteq \N'_\QQ = \N_\QQ \oplus \QQ,
        \]
		together with the colored faces of such cones.
		In \Cref{F2-colored-fan}, we prove that \(\mathbb{F}_2\) is a genuine colored fan.
		
    \item\noindent\phantomsection\label{make-fan-complete}
		Next, we introduce two new rays \(\mathbb{Q}_{\ge0}\rho'(Y^\pm)\coloneq\QQ_{\ge0}((0,\dots,0),\pm1)\subseteq \N'_{\QQ}\) corresponding to two \(G\)-invariant divisors \(Y^+,Y^- \subseteq X'\).
		Set \(\Delta' \coloneqq \Delta \sqcup \{ Y^+, Y^-\}\).
        In \Cref{sec:define-F3}, we construct a complete colored fan \(\mathbb{F}_3\) extending \(\mathbb{F}_2\) and incorporating \(\QQ_{\ge0}\rho'(Y^\pm)\).
        Roughly, the colored fan \(\mathbb{F}_3\) is constructed by introducing a new colored cone \((\mathcal{C}+\QQ_{\ge0}\rho(Y^{+}),\mathcal{F})\) for each \((\mathcal{C},\mathcal{F})\in\mathbb{F}_2\) which is ``visible'' from \(((0,\dots,0),+\infty)\).
        Similarly, we introduce a new colored cone \((\mathcal{C}+\QQ_{\ge0}\rho(Y^{-}),\mathcal{F})\) for each \((\mathcal{C},\mathcal{F})\in\mathbb{F}_2\) which is ``visible'' from \(((0,\dots,0),-\infty)\).
        Further details can be found in \Cref{sec:define-F3}.

        Now replace \(\mathbb{F}_2\) with \(\mathbb{F}_3\).
 
	\item\noindent\phantomsection\label{item:make-colors-rays}
		Next, we modify the colored fan \(\mathbb{F}_3\) so that every color \(D\) which is used in the coloring \(D\in\mathcal{F}\) of a colored cone \((\mathcal{C},\mathcal{F})\), spans a ray of this cone \(\mathcal{C}\).
		Let \(J\subseteq\D'\) be the set of colors for which there exists a colored cone \((\mathcal{C},\mathcal{F})\) of \(\mathbb{F}_3\) such that \(D\in\mathcal{F}\) and \(\QQ_{\geq0}\rho'(D)\) is not a ray of \(\mathcal{C}\).
		In \Cref{sec:colored-stellar-subdiv}, we extend the notion of star subdivisions (or stellar subdivisions) to colored cones.
		With this tool at hand, we now perform the following modifications.
		For every \(D\in J\) add a new ray spanned by \(\rho'(D)\) and perform a star subdivision of \(\mathbb{F}_3\) in the direction \(\QQ_{\ge0}\rho'(D)\).
		This replaces \(\mathbb{F}_3\) with a colored fan where some colored cones are subdivided such that \(\QQ_{\ge0}\rho'(D)\) is a ray of the new subdivided cones.
		Replace \(\mathbb{F}_3\) by the resulting colored fan, and continue this process for the remaining \(D\in J\).
		We denote the final colored fan by \(\mathbb{F}_4\).

    \item\noindent\phantomsection\label{item:final-step}
        Take the colored fan \(\mathbb{F}_4\) from Step~\eqref{item:make-colors-rays} and triangulate every colored cone in \(\mathbb{F}_4\) without adding any new rays.
        Finally, throw away each abstract colored cone in this new colored fan which is not a genuine colored cone.
        The result is a colored fan \(\mathbb{F}_5\) that is complete. 
    \end{enumerate}
    The output of this algorithm is the colored fan \(\mathbb{F}'\coloneq\mathbb{F}_5\) of a spherical variety \(X'\hookleftarrow G/H'\).
\end{algorithm}

In \Cref{sec:proof-of-gorenstein-algorithm}, we prove the correctness of \Cref{Gorenstein algorithm}.

\begin{proof}[Proof of \Cref{reduction to gorenstein}]
	With \Cref{prop:Gorenstein algorithm}, it suffices to show that the spherical variety \(X'\) whose colored fan is the output of \Cref{Gorenstein algorithm} satisfies \(\widetilde{\wp}(X)=\widetilde{\wp}(X')\).
	Firstly, notice that \(\dim{X} - \rk{X}\) is unchanged by Step~\eqref{first-step}:
	\begin{align*}
		\dim{G/H'}-\rk{G/H'} &= (\rk{\M'}+|R^+| - |R_{S^p}^+|) - \rk{\M'}\\
		&= |R^+| - |R_{S^p}^+| = \dim{G/H} - \rk{G/H},
	\end{align*}
	where \(R_{S^p}\) denotes the sub root system of \(R\) spanned by \(S^p\).
    
	Remark that no further steps of \Cref{Gorenstein algorithm} change \(\dim{X} - \rk{X}\).
	Next, let us confirm that Step~\eqref{item:add-kappa} preserves the linear program in \(\widetilde{\wp}(X)\).
	We begin by observing that in the Luna datum of \(G/H'\) only \(\M\) and \(\rho:\Delta\to\N\) change.
	Furthermore, there is a one-to-one correspondence from the rays of \(\mathbb{F}\) and \(\mathbb{F}_2\).
	Hence, there is natural identification of the set of \(B\)-invariant divisors \(\Delta\) in \(X\) and the ones in the spherical embedding of \(G/H'\) given by the colored fan \(\mathbb{F}_2\).
	Let us denote the latter embedding by \(X_2\) and write \(\widetilde{D}\) for the \(B\)-invariant divisor in \(X_2\) corresponding to \(D \in \Delta\).
	With the above observation we have
	\[
		-K_{X_2} = \sum_{D \in \Delta} m_D \widetilde{D}.
	\]
	That is, the coefficients in the anticanonical class do not change.
	Therefore, the polyhedron used in the computation of \(\widetilde{\wp}(X_2)\) is given as follows
	\[
		{Q_2}^* = \bigcap_{D \in \Delta} \{ v \in \M_\QQ \oplus \QQ \mid \langle \rho'(\widetilde{D}), v \rangle \ge -m_D\}.
	\]
	The tailcone of \(G/H'\) is \(\T' = \cone(\Sigma') = \cone(\Sigma) = \T \times \{0\} \subseteq \M_\QQ \oplus \QQ\), which is contained in the first factor of \(\M'_\QQ = \M_\QQ \oplus \QQ\).
	Therefore, the new region of optimisation in \(\widetilde{\wp}(X_2)\) is
	\[
		Q_2^*\cap\T'=(Q^* \cap \T) \times \{0\}.
	\]
	The new objective function is \(\sum_{D \in \Delta'}\rho'(\widetilde{D})\).
	Let \(\pi \colon \N_\QQ' = \N_\QQ \oplus \QQ \to \N_\QQ\) be the projection onto the first factor.
	It is straightforward to verify that \(\sum_{D\in{\Delta}}\rho'(\widetilde{D})\) and \(\pi^*(\sum_{D\in\Delta}\rho(D))\) coincide on \(\M_\QQ \subseteq \M'_\QQ = \M_\QQ\oplus \QQ\), and thus these two linear forms are equal on \(Q_2^* \cap \T'=(Q^*\cap \T) \times \{0\}\).
	From this it immediately follows that the linear program in \(\widetilde{\wp}(X)\) is unchanged by Step~\eqref{item:add-kappa}, and \(\widetilde{\wp}(X)\) is preserved, i.e.\ \(\widetilde{\wp}(X) = \widetilde{\wp}(X_2)\).

	In Step~\eqref{make-fan-complete} the new rays \(\QQ_{\ge0}\rho(Y^\pm)\) lead to the set of \(B\)-invariant divisors \(\Delta' = \Delta \sqcup \{ Y^+, Y^-\}\) in the spherical embedding \(G/H' \hookrightarrow X_3\) given by the colored fan \(\mathbb{F}_3\).
	Then, the anticanonical class of \(X_3\) is given by 
	\[
		-K_{X_3} = \sum_{D \in \Delta} m_D \overline{D} + Y^+ + Y^-
	\]
	where we have written a bar over each \(B\)-invariant divisor in \(X_3\) that corresponds to a \(B\)-invariant divisor in \(X_2\).
	For simplicity, we slightly abuse notation and simply write \(D\) (\(D \in \Delta'\)) for each \(B\)-invariant prime divisor \(D\) in \(X_3\).
	Notice that the divisors \(Y^\pm\) are \(G\)-invariant, and thus they satisfy \(m_{Y^\pm}=1\).
	In particular, the sum \(\sum_{D\in\Delta'}(m_D-1)\) appearing in the computation of \(\widetilde{\wp}(X_3)\) coincides with the sum \(\sum_{D\in\Delta}(m_D-1)\) appearing in the computation of \(\widetilde{\wp}(X)\) (resp.~\(\widetilde{\wp}(X_2)\)).
	The polyhedron used in \(\widetilde{\wp}(X_3)\) is given by
	\[
		Q_3^* = Q_2^* \cap \{ v \in \M_\QQ \oplus \QQ \mid \langle \rho'(Y^+), v\rangle \ge -1\} \cap \{ v \in \M_\QQ \oplus \QQ \mid \langle \rho'(Y^-), v\rangle \ge -1\},
	\]
	and thus \(Q_3^* \cap \T' = Q_2^*\cap \T' = (Q^*\cap \T) \times \{0\}\).
	Furthermore, the linear forms \(\rho'(Y^\pm)\) vanish on \(\M_\QQ \subseteq \M'_\QQ = \M_\QQ \oplus \QQ\), and thus adding them does not change the objective function in \(\widetilde{\wp}(X_3)\) within the optimisation region, i.e.\
	\[
		\sum_{D\in\Delta'}\rho'(D)|_{(Q^*\cap\T) \times \{0\}} = \sum_{D \in \Delta} \rho'(D)|_{(Q^*\cap \T)\times\{0\}} = \left.\pi^*\left(\sum_{D \in \Delta}\rho(D)\right)\right|_{(Q^*\cap\T) \times\{0\}}.
	\]
	Therefore, the linear program in \(\widetilde{\wp}(X)\) is preserved when the rays \(\rho(Y^\pm)\) are added.
	Hence, \(\widetilde{\wp}(X)\) is preserved by Step~\eqref{make-fan-complete}.

    It is apparent from its definition that \(\widetilde{\wp}\) only depends on the \(B\)-invariant divisors of a spherical embedding, and not on the higher dimensional modifications which are made to the colored fan \(\mathbb{F}_3\) in Steps~\eqref{item:make-colors-rays}~and~\eqref{item:final-step}.
    Therefore, \(\widetilde{\wp}(X)\) is preserved by \Cref{Gorenstein algorithm}, and we have \(\widetilde{\wp}(X')=\widetilde{\wp}(X)\).
\end{proof}

\begin{prop}\noindent\phantomsection\label{complete P theorem} Let \(X\) be a complete spherical variety, then 
	\[
		\widetilde{\wp}(X)\ge0.
	\]
\end{prop}

\begin{proof}
    Let \(X\) be a complete spherical variety.
	Then, by \Cref{reduction to gorenstein} and \Cref{Spherical Mukai}, there is a complete \(\QQ\)-Gorenstein spherical variety \(X'\) such that 
    \[
        \widetilde{\wp}(X)=\widetilde{\wp}(X')\ge 0.\qedhere
    \]
\end{proof}

%--------------------------------------------------------------------------------
\section{Proof of Algorithm~\ref{Gorenstein algorithm}}\noindent\phantomsection\label{sec:proof-of-gorenstein-algorithm}
In this section, we prove the correctness of \Cref{Gorenstein algorithm}.
Firstly, we study the ``colored analogue'' of star subdivisions of fans from toric geometry to complete Step~\eqref{item:make-colors-rays} of \Cref{Gorenstein algorithm}.
The section is concluded with the proof of \Cref{Gorenstein algorithm}.

\subsection{A colored analogue of star subdivisions}\noindent\phantomsection\label{sec:colored-stellar-subdiv} Before proving the correctness of \Cref{Gorenstein algorithm}, we want to give a brief overview of star subdivisions of colored fans.
There are many equivalent ways to define the star subdivision of a (colored) fan.
One way to introduce these modifications of fans is via the so-called ``star'' and ``link'' of a ray or cone contained in the support of a fan.
However, for our setting the following more straightforward (yet less systematic) approach seems more useful.

\begin{defin}\noindent\phantomsection\label{def:star-subdiv}
	Let \(\mathbb{F}\) be a colored fan.
	Denote by \(\mathbb{F}^a\) the corresponding abstract colored fan, obtained by extending \(\mathbb{F}\) with abstract colored faces of colored cones in \(\mathbb{F}\).
	Notice that \(\mathbb{F}^a\) still satisfies the conditions of a colored fan, except the relative interiors of abstract colored cones might not intersect the valuation cone.

	Following~\cite[Section~11]{ToricBook}, we introduce for a given color \(D \in \D\) with \(\rho(D) \in |\mathbb{F}| = |\mathbb{F}^a|\), the collection \(\mathbb{F}^a(D)\) consisting of the following cones:
	\begin{itemize}
	\item
		\((\mathcal{C},\mathcal{F})\), where \((\mathcal{C}, \mathcal{F}) \in \mathbb{F}^a\) with \(\rho(D) \not \in \mathcal{C}\).
	\item
		\((\cone(\mathcal{C}',\rho(D)), \mathcal{F}'\cup \{D\})\), where \((\mathcal{C}',\mathcal{F}')\in \mathbb{F}^a\) with \(\rho(D) \not \in \mathcal{C}'\) and \(\{\rho(D)\} \cup \mathcal{C}' \subseteq \mathcal{C}\) for some \((\mathcal{C}, \mathcal{F}) \in \mathbb{F}^a\).
	\end{itemize}
	The set \(\mathbb{F}^*(D)\) of genuine colored cones in \(\mathbb{F}^a(D)\) is called the \emph{star subdivision} of \(\mathbb{F}\) at \(D\).
\end{defin}

\begin{rem}\noindent\phantomsection\label{rem:why-include-abstr-faces}
	Notice that in \Cref{def:star-subdiv} we need to include abstract colored faces.
	For example, consider the following setup: \(\V \coloneqq \{ x+y \le 0\} \subseteq \QQ^2\), \(\D = \{ D_1, D_2, D_3\}\), and
	\[
		\rho\colon \D \to \QQ^2; \quad\begin{cases}
			\rho(D_1)=(1,0)\\
			\rho(D_2)=(0,1)\\
			\rho(D_3)=(-2,1).
		\end{cases}
	\]
	We remark that this ``colored data'' does not arise as a genuine Luna datum of a spherical variety.
	However, in higher dimensions, a phenomenon we now describe may occur.
	For illustration purposes, we focus on \(2\) dimensions.

	Consider the colored fan \(\mathbb{F}\) that consists of all (genuine) colored faces of the colored cone \((\mathcal{C},\mathcal{F}) \coloneqq (\cone((-3,1),(1,0)), \{D_1,D_2,D_3\})\) (see \Cref{fig:why-include-abstr-faces}).
	\begin{figure}[!ht]
		\begin{tikzpicture}
			\clip (-2.5,-0.5) rectangle (1.5,1.5);
			\fill[fill=gray,opacity=.8] (0,0) -- (-6,2) -- (-6,2.5) -- (1.5,2.5) -- (1.5,0) -- cycle;
			\fill[opacity=.2,fill=gray] (-2.5,2.5) -- (2.5,-2.5) -- (-2.5,-2.5) -- cycle;
			
			\draw[thick] (0,0) -- (-6,2);
			\draw[thick] (0,0) -- (1.5,0);

			\draw[-latex] (-2.5,0) -- (1.5,0);
			\draw[-latex] (0,-.5) -- (0,1.5);

			\node at (-2.5,-0.5) [above right] {\(\V\)};
			\draw[very thin] (-2.5,2.5) -- (2.5,-2.5);

			\draw (1,0) circle (2pt) node[below] {\tiny\(\rho(D_1)\)};
			\draw (0,1) circle (2pt) node[right] {\tiny\(\rho(D_2)\)};
			\draw (-2,1) circle (2pt) node[above] {\tiny\(\rho(D_3)\)};
		\end{tikzpicture}
		\caption{Illustration of \Cref{rem:why-include-abstr-faces}.\label{fig:why-include-abstr-faces}}
	\end{figure}
	That is,
	\[
		\mathbb{F} = \{(\mathcal{C},\mathcal{F}), (\cone((-3,1)),\{\}), (\{0\},\{\})\}.
	\]
	However, \(\mathbb{F}^a\) also includes the abstract colored face \((\cone((1,0)), \{D_1\})\).

	It is straightforward to verify that \(\mathbb{F}^*(D_3)\) has two maximal cones, namely
	\begin{align*}
		(\mathcal{C}_1, \mathcal{F}_1) &= (\cone((-3,1),\rho(D_3)), \{ D_3 \}), \quad \text{and}\\
		(\mathcal{C}_2, \mathcal{F}_2) &= (\cone(\rho(D_1),\rho(D_3)), \{D_1,D_3\}).
	\end{align*}
	The latter cone arises as the cone generated by the abstract face \((\cone(\rho(D_1)), \{D_1\})\).
	If we did not allow abstract faces this cone would not appear, and thus the support of \(\mathbb{F}\) and \(\mathbb{F}^*(D)\) would differ.
	This doesn't happen with our definition as we show in \Cref{prop:star-subdiv}.
\end{rem}

\begin{prop}\noindent\phantomsection\label{prop:star-subdiv}
	Let \(\mathbb{F}\) be a colored fan and let \(D \in \D\) with \(\rho(D) \in |\mathbb{F}|\).
	Then \(\mathbb{F}^*(D)\) is a (genuine) colored fan with \(|\mathbb{F}|\cap \V = |\mathbb{F}^*(D)| \cap \V\).
\end{prop}
\begin{proof}
	From \Cref{def:star-subdiv}, it is clear that each colored cone \((\mathcal{C},\mathcal{F}) \in \mathbb{F}^*(D)\) is contained in some colored cone \((\mathcal{C}', \mathcal{F'}) \in \mathbb{F}^a\).
	Notice that the collection of colored cones in \(\mathbb{F}^a(D)\) that are contained in a fixed colored cone \((\mathcal{C}',\mathcal{F}')\) of \(\mathbb{F}^a\) can be considered a star subdivision of the fan \(\mathbb{F}_{\mathcal{C}'}^a\) arising from the cone \((\mathcal{C}',\mathcal{F}')\) by adding abstract colored faces.
	Here, we think of the star subdivision as trivial (i.e.\ doesn't change the fan) if the colored cone \((\mathcal{C}',\mathcal{F}')\) does not contain \(\rho(D)\).
	Hence, it suffices to consider colored fans \(\mathbb{F}_{\mathcal{C}'}\) associated to a colored cone \((\mathcal{C}',\mathcal{F}')\) with \(\rho(D) \in \mathcal{C}'\).

	In this case, it follows from the corresponding definition of star subdivision in toric geometry that \(\{\mathcal{C} \mid (\mathcal{C},\mathcal{F}) \in \mathbb{F}_{\mathcal{C}'}^a(D)\}\) is a fan (in the toric sense) subdividing the cone \(\mathcal{C}'\).
	In particular, each \(\nu \in \mathcal{C}'\) is contained in the relative interior of exactly one cone \((\mathcal{C},\mathcal{F}) \in \mathbb{F}_{\mathcal{C}'}^a(D)\).
	Removing faces in \(\mathbb{F}_{\mathcal{C}'}^a(D)\) preserves this property.

	From this observation it is straightforward to verify that each \(\nu \in \V\) is contained in the relative interior of at most one colored cone of \(\mathbb{F}^*(D)\).
	Indeed, suppose we have \((\mathcal{C}_1,\mathcal{F}_1), (\mathcal{C}_2,\mathcal{F}_2) \in \mathbb{F}^*(D)\) such that \(\nu \in \mathcal{C}_1^\circ \cap \mathcal{C}_2^\circ\).
	Since, \(\mathbb{F}\) is a colored fan and \(\nu \in \mathcal{C}_1^\circ \subseteq |\mathbb{F}|\), it follows that there exists exactly one colored cone \((\mathcal{C}',\mathcal{F}') \in \mathbb{F}\) which contains \(\nu\) in its relative interior.
	From the definition of star subdivision, there exists \((\widetilde{\mathcal{C}}_i,\widetilde{\mathcal{F}}_i) \in \mathbb{F}^a\) such that \(\mathcal{C}_i \subseteq \widetilde{\mathcal{C}}_i\).
	Since \(\nu \in {(\mathcal{C}')}^\circ\) and \(\mathbb{F}^a\) is a fan in the toric sense when restricted to \(\V\), it follows that \(\mathcal{C'} \preceq \widetilde{\mathcal{C}}_i\), i.e.\ \(\mathcal{C}'\) is a colored face of \(\widetilde{\mathcal{C}}_i\).
	Clearly, \(\mathbb{F}_{\mathcal{C}'}^a(D)\) is a subfan of \(\mathbb{F}_{\widetilde{\mathcal{C}}_i}^a(D)\).
	Since \(\nu \not \in |\mathbb{F}^a_{\widetilde{\mathcal{C}}_i}(D)| \setminus |\mathbb{F}^a_{\mathcal{C}'}(D)| = \widetilde{\mathcal{C}}_i\setminus \mathcal{C}'\), it follows that \(\nu \in |\mathbb{F}_{\mathcal{C}'}^a(D)|\), and thus \(\mathcal{C}_i \in \mathbb{F}_{\mathcal{C}'}^a(D)\) as it is the unique cone in \(\mathbb{F}^a_{\widetilde{\mathcal{C}}_i}(D)\) which contains \(\nu\) in its relative interior.
	Similarly, one shows that \(\mathcal{C}_2 \in \mathbb{F}_{\mathcal{C}'}^a(D)\).
	The statement follows from the observation that \(\mathbb{F}^a_{\mathcal{C}'}(D)\) is a fan in the toric sense, and thus \(\mathcal{C}_1 = \mathcal{C}_2\).

	Finally, notice that if \(\nu \in \V \cap |\mathbb{F}|\) then there exists exactly one cone \((\mathcal{C}',\mathcal{F}') \in \mathbb{F}\) such that \(\nu \in {(\mathcal{C}')}^\circ\).
	Then in the star subdivision \(\mathbb{F}^a_{\mathcal{C}'}(D)\) of the colored fan \(\mathbb{F}^a_{\mathcal{C}'}\) obtained from the colored cone \((\mathcal{C}',\mathcal{F}')\), there is exactly one \((\mathcal{C},\mathcal{F})\) that contains \(\nu\) in its relative interior.
	Clearly, \((\mathcal{C},\mathcal{F}) \in \mathbb{F}^a(D)\).
    
	However, then clearly \(\mathcal{C}^\circ \cap \V\neq \emptyset\), and thus \((\mathcal{C},\mathcal{F}) \in \mathbb{F}^*(D)\).
	This proves that \(|\mathbb{F}|\cap \V \subseteq |\mathbb{F}^*(D)|\cap \V\).
	Conversely, by \Cref{def:star-subdiv}, 
	\[
		|\mathbb{F}^*(D)|\cap \V \subseteq |\mathbb{F}^a|\cap \V = |\mathbb{F}| \cap \V.\qedhere
	\]
\end{proof}

%--------------------------------------------------------------------------------
\subsection{Proving Algorithm~\ref{Gorenstein algorithm}}\noindent\phantomsection\label{sec:define-F3}
In this section, we prove \Cref{prop:Gorenstein algorithm}, which asserts the correctness of \Cref{Gorenstein algorithm}.
Throughout, let \(\pi\colon\N_\QQ'=\N_\QQ\oplus\QQ\to\N_\QQ\) be the projection onto the first factor, and write \(n\coloneq\dim \N_\QQ\).
Furthermore, we will denote the additional rays that were introduced in Step~\eqref{make-fan-complete} by \(\pm\rho = \QQ_{\ge0} (0,\pm 1) \subseteq \N'_\QQ = \N_\QQ \oplus \QQ\).
Finally, recall that in Step~\eqref{item:add-kappa} of \Cref{Gorenstein algorithm}, we introduced cones in \(\mathbb{F}_2\) arising from cones \((\mathcal{C},\mathcal{F}) \in \mathbb{F}\) which are denoted by the same symbols with an added apostrophe, i.e.\ \((\mathcal{C}',\mathcal{F}')\).

\begin{prop}\noindent\phantomsection\label{prop:Gorenstein algorithm}
    The output of \Cref{Gorenstein algorithm} is the colored fan \(\mathbb{F}'\) of a complete \(\QQ\)-Gorenstein spherical variety.
\end{prop}

We now establish some intermediate results, so that \Cref{prop:Gorenstein algorithm} follows as a corollary.

Firstly, we need the following elementary, but important, observation.

\begin{prop}\noindent\phantomsection\label{prop:preimage-pi}
	Let \((\mathcal{C},\mathcal{F}) \in \mathbb{F}\).
	Then 
	\[
		\pi^{-1}(\mathcal{C}) \cap |\mathbb{F}_2| = \mathcal{C}'.
	\]
\end{prop}
\begin{proof}
	The inclusion ``\(\supseteq\)'' is clear.

	Let \((\sigma,\varphi) \in \mathbb{F}_2\) be such that \(\pi(\sigma) \subseteq \mathcal{C}\).
	By definition of \(\mathbb{F}_2\), there exists \((\mathcal{C}_0,\mathcal{F}_0) \in \mathbb{F}\) such that \((\sigma,\varphi) \preceq (\mathcal{C}'_0,\mathcal{F}'_0)\).
	Since \(\pi(\sigma) \subseteq \mathcal{C}\cap\mathcal{C}_0\), a straightfoward argument shows that \(\pi(\sigma) \subseteq \tau\) for some colored cone \((\tau,\gamma) \in \mathbb{F}\) with \(\tau\preceq \mathcal{C}\) and \(\tau \preceq \mathcal{C}_0\).

    For a colored cone \((\mathcal{C}_1,\mathcal{F}_1)\in\mathbb{F}\), let us write \(\Delta(\mathcal{C}_1,\mathcal{F}_1)\) for the set of \(B\)-invariant divisors associated to the colored cone \((\mathcal{C}_1,\mathcal{F}_1)\).
    Then for all \(x\in\sigma\), there is an expression of the form \(x=\sum_{D\in\Delta(\mathcal{C}_0,\mathcal{F}_0)}\lambda_D\rho'(D)\) for \(\lambda_D\ge0\). 
    So, \(\pi(x)=\sum_{D\in\Delta(\mathcal{C}_0,\mathcal{F}_0)}\lambda_D\rho(D)\in\tau\).
    Since \(\tau\preceq\mathcal{C}_0\), we get \(\lambda_D=0\) for \(D\in\Delta(\mathcal{C}_0,\mathcal{F}_0)\setminus \Delta(\tau,\gamma)\).
    Therefore \(x=\sum_{D\in\Delta(\tau,\gamma)}\lambda_D\rho'(D)\in\tau'\subseteq\mathcal{C}'\).
\end{proof}

\begin{prop}\noindent\phantomsection\label{F2-colored-fan}
    \(\mathbb{F}_2\) is a colored fan.
\end{prop}

\begin{proof} 
	We show that for colored cones \((\sigma_1,\varphi_1),(\sigma_2,\varphi_2)\in\mathbb{F}_2\) such that there exists a point \(x\in\sigma_1^\circ\cap\sigma_2^\circ\cap\V'\neq\emptyset\), we have \(\sigma_1 = \sigma_2\).
	Let us first recall that the colored cones of \(\mathbb{F}_2\) are \((\sigma,\varphi)\preceq(\mathcal{C}',\mathcal{F}')\) for \((\mathcal{C}',\mathcal{F}')\) a genuine colored cone arising from \((\mathcal{C},\mathcal{F})\in\mathbb{F}\).
	So there exist colored cones \((\mathcal{C}_i,\mathcal{F}_i)\in\mathbb{F}\) such that \((\sigma_i,\varphi_i)\preceq(\mathcal{C}'_i,\mathcal{F}'_i)\) for \(i=1,2\).

	Since \(\mathbb{F}\) is a colored fan, there is a colored cone \((\mathcal{C},\mathcal{F})\in\mathbb{F}\) with \(\pi(x)\in\mathcal{C}^\circ\).
	Moreover, since \(\pi(\sigma_i^\circ)={\pi(\sigma_i)}^\circ\), we have \(\pi(x)\in{\pi(\sigma_1)}^\circ\cap{\pi(\sigma_2)}^\circ\subseteq\mathcal{C}_1\cap\mathcal{C}_2\).
	From this, it is straightfoward to show that \((\mathcal{C},\mathcal{F})\preceq(\mathcal{C}_1,\mathcal{F}_1),(\mathcal{C}_2,\mathcal{F}_2)\), and \(\pi(\sigma_i) \subseteq\mathcal{C}\).
	We now show that the face property also holds for the corresponding colored cones in \(\mathbb{F}_2\), i.e.\ \((\mathcal{C}',\mathcal{F}')\preceq(\mathcal{C}_1',\mathcal{F}_1'),(\mathcal{C}_2',\mathcal{F}_2')\).
	Indeed, there are linear functionals \(f_i\in\mathcal{C}_i^\vee\) which select the face \(\mathcal{C}\): that is \(\{f_i=0\}\cap\mathcal{C}_i=\mathcal{C}\). 
Then it follows from \Cref{prop:preimage-pi} that \((f_i,0)\in\mathcal{C}_i'^\vee\) selects \(\mathcal{C}'\) as a face of \(\mathcal{C}_i'\).

	Recall that \(\pi(\sigma_i)\subseteq\mathcal{C}\). Then by \Cref{prop:preimage-pi}, we have \(\sigma_i\subseteq\mathcal{C}'\). 
	Using \((\sigma_i,\varphi_i) \preceq (\mathcal{C}'_i,\mathcal{F}'_i)\), and \((\mathcal{C}',\mathcal{F}')\preceq(\mathcal{C}_1',\mathcal{F}_1'),(\mathcal{C}_2',\mathcal{F}_2')\), it follows that \((\sigma_1,\varphi_1),(\sigma_2,\varphi_2)\) are faces of \((\mathcal{C}',\mathcal{F}')\).
	As the faces of \((\mathcal{C}',\mathcal{F}')\) form a colored fan, from \(x\in\sigma_1^\circ\cap\sigma_2^\circ\cap\V'\) we deduce that \(\sigma_1=\sigma_2\).
\end{proof}

To prove that \(\mathbb{F}_3\) is a colored fan, we use \emph{vector configurations}, for which we briefly recall the necessary background.
We use~\cite{DLRSTriangulations2010} as a general reference of this well-known topic.

\begin{defin}[{\cite[Section~2.5]{DLRSTriangulations2010}}]
    A \emph{vector configuration} is a multi-set of vectors:
    \[
        \mathbf{A}\coloneq(v_i\,|\,i\in S)\subseteq\RR^n.
    \]
	For each linear functional \(\psi\in{(\RR^n)}^*\) with \(\psi(\mathbf{A})\ge0\), the \emph{face of \(\mathbf{A}\) in direction \(\psi\)} is the maximal subconfiguration \(S\subseteq\mathbf{A}\) on which \(\psi(S)=0\).
    Write \(F\le S\) if \(F\) is a face of \(S\).
    
    A \emph{polyhedral subdivision} of a vector configuration \(\mathbf{A}\) is a collection of subconfigurations \(\mathscr{S}\), called \emph{cells}, satisfying the following properties:
    \begin{enumerate}
    \item
        For all \(C\in \mathscr{S}\), if \(F\le C\) then \(F\in \mathscr{S}\).
    \item
        \(\bigcup_{C\in\mathscr{S}}\cone(C)\supseteq\cone A\).
    \item
        \(\mathcal{C}^\circ\cap\mathcal{C}'^\circ\neq\emptyset\) for \(\mathcal{C},\mathcal{C}'\in\mathscr{S}\) implies that \(\mathcal{C}=\mathcal{C}'\).
    \end{enumerate}
    
    A \emph{lift} of a vector configuration \(\mathbf{A}\) is a function \(\omega\colon \mathbf{A}\to\RR\), which associates a \emph{height} to each vector in \(\mathbf{A}\).
    A lift gives rise to a \emph{lifted} vector configuration \(\mathbf{A}^\omega\coloneq\{(v_i,\omega(v_i))\mid i\in S\}\).
    
    Given a lift \(\omega\colon \mathbf{A}\to\RR\), the \emph{lower faces} \(\mathbf{A}^l\) of \(\mathbf{A}^\omega\) are the faces of \(\mathbf{A}^\omega\) in the direction of a linear functional \(\psi\) which is positive on the last coordinate.
    Similarly, the \emph{upper faces} \(\mathbf{A}^u\) of \(\mathbf{A}^\omega\) are the faces of \(\mathbf{A}^\omega\) in the direction of a linear functional \(\psi\) which is negative on the last coordinate.
\end{defin}

\begin{rem}\noindent\phantomsection\label{rem:equiv-descr-lower-upper-faces}
	By using \(\sigma^{\vee\vee} = \sigma\) for each polyhedral cone \(\sigma \subseteq \N'_\QQ\), it is straightfoward to show that \(F \le \mathbf{A}^\omega\) is a lower face (resp.~upper face) if and only if \(x-\lambda v_\rho \not \in \cone(\mathbf{A}^\omega)\) (resp.~\(x+\lambda v_\rho \not \in \cone(\mathbf{A}^\omega)\)) for each $x \in \cone(F)$ and $\lambda>0$.
	Here, \(v_\rho\) denotes a ray generator of the ray \(\rho\).
\end{rem}

We require the following result:

\begin{lemma}[{\cite[Lemma~2.5.11]{DLRSTriangulations2010}}]\noindent\phantomsection\label{rem:subdivision}
	Let \(\omega\colon\mathbf{A}\to\RR\) be a lift of a vector configuration \(\mathbf{A}\).
	If there exists a linear function \(\ell\colon \RR^n\to\RR\) such that \((\omega+\ell)(\mathbf{A})\ge0\), then the lower faces \(\mathbf{A}^l\) of \(\mathbf{A}^\omega\) give rise to a polyhedral subdivision of \(\mathbf{A}\), whose cells are \(\pi(C)\) for \(C\) a cell of \(\mathbf{A}^l\).
	Here, \(\pi\) is the projection which forgets the last coordinate.
\end{lemma}

The relation of vector configurations to our setting is that colored cones can naturally be interpretated as vector configurations, where in general, it is necessary to consider multisets.

\begin{defin}
    We associate the following vector configuration to a colored cone \((\mathcal{C},\mathcal{F})\):
    \[
        \mathbf{A}_{(\mathcal{C},\mathcal{F})}\coloneq\{v_1,\dots,v_{r_\mathcal{C}}\}\cup(\rho(D)\mid D\in\mathcal{F}),
    \]
    where the primitive ray generators of rays of \(\mathcal{C}\)  corresponding to \(G\)-invariant divisors are \(\{v_1,\dots,v_{r_\mathcal{C}}\}\), and \((\rho(D)\mid D\in\mathcal{F})\) is considered as a multiset.
\end{defin}

Then the cones in \(\mathbb{F}_2\) accept a natural interpretation as lifted vector configurations coming from cones in \(\mathbb{F}\).

\begin{defin}
    For each colored cone \((\mathcal{C},\mathcal{F})\in\mathbb{F}\), consider \((\mathcal{C}',\mathcal{F}')\in\mathbb{F}_2\).
    Its associated vector configuration \(\mathbf{A}_{(\mathcal{C}',\mathcal{F}')}\) is a lifting of \(\mathbf{A}_{(\mathcal{C},\mathcal{F})}\) with height function \(\omega\colon \mathbf{A}_{(\mathcal{C},\mathcal{F})}\to\RR\), defined by \(\omega(v_i)=1\) for \(v_i\) the primitive ray generator of the ray of a \(G\)-invariant divisor, and \(\omega(\rho(D))=m_D\ge1\) for \(D\in\mathcal{F}\).
\end{defin}

In order to define \(\mathbb{F}_3\), we note the following subdivisions of the colored cones of \(\mathbb{F}\):

\begin{prop}\label{prop:subdivision-of-cone}
    For each \((\mathcal{C},\mathcal{F})\in\mathbb{F}\), there are two polyhedral subdivisions of \((\mathcal{C},\mathcal{F})\) into abstract colored cones: \(\mathscr{S}^u(\mathcal{C},\mathcal{F})\) and \(\mathscr{S}_l(\mathcal{C},\mathcal{F})\), which in the sense of \Cref{rem:subdivision}, arise from the upper and lower faces of \(\mathbf{A}_{(\mathcal{C}',\mathcal{F}')}\) respectively.
\end{prop}

\begin{proof}
    Since \(\omega\colon \mathbf{A}_{(\mathcal{C},\mathcal{F})}\to\RR\) is positive, by \Cref{rem:subdivision} the lower faces of \(\mathbf{A}_{(\mathcal{C}',\mathcal{F}')}\) give rise to a polyhedral subdivision \(\mathscr{S}_l(\mathcal{C},\mathcal{F})\) of \(\mathbf{A}_{(\mathcal{C},\mathcal{F})}\).
	Similarly, as \(\mathcal{C}\) is strictly convex, there is a linear function \(\ell\) which is strictly positive on \(\mathcal{C}\setminus\{0\}\).
    So applying \Cref{rem:subdivision} to the height function \(-\omega+\lambda\ell\) for a sufficiently large \(\lambda>0\), the upper faces of \(\mathbf{A}_{(\mathcal{C}',\mathcal{F}')}\) with respect to \(\omega\) give rise to a polyhedral subdivision \(\mathscr{S}^u(\mathcal{C},\mathcal{F})\) of \(\mathbf{A}_{(\mathcal{C},\mathcal{F})}\). 

    To each cell \(S\) in the subdivisions \(\mathscr{S}^u(\mathcal{C},\mathcal{F})\), \(\mathscr{S}_l(\mathcal{C},\mathcal{F})\), we naturally associate the abstract colored cone generated by \(S\), which inherits the colorings of the vectors in \(S\).
\end{proof}

\begin{rem}
	For a colored cone \((\sigma^u,\varphi^u) \in \mathscr{S}^u(\mathcal{C},\mathcal{F})\), we write \(((\sigma^u)',\varphi^u)\) to denote the corresponding upper face of \(\mathbf{A}_{(\mathcal{C}',\mathcal{F}')}\).
	Similarly, for the lower faces of \(\mathbf{A}_{(\mathcal{C}',\mathcal{F}')}\) we write \((\sigma'_l,\varphi_l)\) if \((\sigma_l,\varphi_l) \in \mathscr{S}_l(\mathcal{C},\mathcal{F})\).
\end{rem}

Let \((\mathcal{C},\mathcal{F}),(\widetilde{\mathcal{C}},\widetilde{\mathcal{F}})\in\mathbb{F}\) be colored cones such that \((\mathcal{C},\mathcal{F})\preceq(\widetilde{\mathcal{C}},\widetilde{\mathcal{F}})\). The following \Cref{prop:compatible-subdivision} asserts that the subdivisions \(\mathscr{S}^u(\mathcal{C},\mathcal{F})\) and \(\mathscr{S}^u(\widetilde{\mathcal{C}},\widetilde{\mathcal{F}})\) are ``compatible''.

\begin{prop}\noindent\phantomsection\label{prop:compatible-subdivision}
	Let \((\mathcal{C},\mathcal{F}) \in \mathbb{F}\) and let \((\mathcal{C}',\mathcal{F}')\) be the lifted cone in \(\mathbb{F}_2\).
	Suppose there exists a face \((\widetilde{\mathcal{C}},\widetilde{\mathcal{F}}) \preceq (\mathcal{C},\mathcal{F})\) and \((\sigma,\varphi) \in \mathscr{S}^u(\mathcal{C},\mathcal{F})\) with \(\sigma' \subseteq \widetilde{\mathcal{C}}'\).
	Then \(\sigma \in \mathscr{S}^u(\widetilde{\mathcal{C}},\widetilde{\mathcal{F}})\).
\end{prop}
\begin{proof}
	This is clear from \Cref{rem:equiv-descr-lower-upper-faces} and \Cref{prop:preimage-pi}.
\end{proof}

\begin{defin}\noindent\phantomsection\label{def:F3-colored-cones}
    We may now define the collection of colored cones \(\mathbb{F}_3\); it consists of three types of colored cones. Firstly, those in \(\mathbb{F}_2\):
    \begin{enumerate}[label=(\arabic*)]
    \item
        \((\mathcal{C},\mathcal{F})\in\mathbb{F}_2\).
    \end{enumerate}
	And secondly, for each \((\mathcal{C},\mathcal{F})\in\mathbb{F}\), the following genuine colored cones:
    \begin{enumerate}[label=(\arabic*),start=2]
    \item
        \(((\sigma^u)'+\rho,\varphi^u)\) for \((\sigma^u,\varphi^u)\) a genuine colored cone in \(\mathscr{S}^u(\mathcal{C},\mathcal{F})\).
    \item
        \((\sigma_l'+(-\rho),\varphi_l)\) for \((\sigma_l,\varphi_l)\) a genuine colored cone in \(\mathscr{S}_l(\mathcal{C},\mathcal{F})\).
    \end{enumerate}
\end{defin}

\begin{prop}
    \(\mathbb{F}_3\) is a colored fan.
\end{prop}

\begin{proof}
	From the definition it is clear that each colored cone in the collection \(\mathbb{F}_3\) is a genuine colored cone, and that a colored face of a colored cone in \(\mathbb{F}_3\) is contained in \(\mathbb{F}_3\).

    It remains to show that two colored cones \((\mathcal{C}_1,\mathcal{F}_1)\), \((\mathcal{C}_2,\mathcal{F}_2)\) in \(\mathbb{F}_3\) which satisfy \(\mathcal{C}_1^\circ\cap\mathcal{C}_2^\circ\cap\V'\neq\emptyset\) must coincide.
    We verify this for all pairs of types of colored cones in \(\mathbb{F}_3\), as appearing in \Cref{def:F3-colored-cones}; we write \((A,B)\) to signify the case where the two cones are of type \(A\) and of type \(B\) respectively.

    \begin{enumerate}
    \item[(1,1)]
        Let \((\mathcal{C}_1,\mathcal{F}_1),(\mathcal{C}_2,\mathcal{F}_2)\in\mathbb{F}_2\).
		Then \(\mathcal{C}_1^\circ\cap\mathcal{C}_2^\circ\cap\V'\neq\emptyset\) implies that \((\mathcal{C}_1,\mathcal{F}_1)=(\mathcal{C}_2,\mathcal{F}_2)\), as \(\mathbb{F}_2\) is a colored fan.
    
    \item[(1,2)]
		Let \((\mathcal{C}_1,\mathcal{F}_1) \in \mathbb{F}_2\), and let \((\sigma^u)'\in\mathbb{F}_2\) be an upper face such that \((\mathcal{C}_2,\mathcal{F}_2)\coloneq((\sigma^u)'+\rho,\mathcal{F}_2)\).
		These colored cones are distinct.
		As \((\sigma^u)'\) is an upper face, by \Cref{rem:equiv-descr-lower-upper-faces}, \({((\sigma^u)'+\rho)}^\circ\cap|\mathbb{F}_2|=\emptyset\).
		Since \(\mathcal{C}_1^\circ\subseteq|\mathbb{F}_2|\), we get \(\mathcal{C}_1^\circ\cap\mathcal{C}_2^\circ\cap\V'=\emptyset\).       
    
    \item[(1,3)]
        This is similar to (1,2).
    
    \item[(2,2)]
		Let \((\sigma_1^u)',(\sigma_2^u)'\) be upper faces of colored cones in \(\mathbb{F}_2\), so that \((\mathcal{C}_i,\mathcal{F}_i)=((\sigma_i^u)'+\rho,\mathcal{F}_i)\).
		Suppose there is some \(x\in{((\sigma_1^u)'+\rho)}^\circ\cap{((\sigma_2^u)'+\rho)}^\circ\cap\V'\).
		Then \(x=\tilde{x}_1+\lambda_1\nu_\rho=\tilde{x}_2+\lambda_2\nu_\rho\) for \(\tilde{x}_1\in(\sigma_1^u)'^\circ, \tilde{x}_2\in(\sigma_2^u)'^\circ\) and without loss of generality \(\lambda_2\ge \lambda_1>0\).
		Here, we write \(\nu_\rho\) for a ray generator of \(\rho\).
		Then since \((\sigma_1^u)'^\circ\) is an upper face, by \Cref{rem:equiv-descr-lower-upper-faces} and \(\tilde{x}_1=\tilde{x}_2+(\lambda_2-\lambda_1)\nu_\rho\), we have \(\lambda_1=\lambda_2\eqcolon\lambda\).
		Therefore, \(x-\lambda\nu_\rho\in(\sigma_1^u)'^\circ\cap(\sigma_2^u)'^\circ\cap\V'\).
		Since \(\mathbb{F}_2\) is a colored fan, \((\sigma_1^u)'=(\sigma_2^u)'\) and \((\mathcal{C}_1,\mathcal{F}_1)=(\mathcal{C}_2,\mathcal{F}_2)\).
        
    \item[(3,3)]
        This is similar to (2,2).
    
    \item[(2,3)]
		Suppose \((\sigma^u)',\sigma_l'\) are upper, respectively lower, faces of colored cones in \(\mathbb{F}_2\), so that \((\mathcal{C}_1,\mathcal{F}_1)=((\sigma^u)'+\rho,\mathcal{F}_1)\) and \((\mathcal{C}_2,\mathcal{F}_2)=(\sigma_l'+(-\rho),\mathcal{F}_2)\).
        Furthermore, suppose that \(x\in\mathcal{C}_1\cap\mathcal{C}_2\cap\V'\neq\emptyset\).

		By definition of \(\mathbb{F}_3\), there are colored cones \((\tau_1,\varphi_1),(\tau_2,\varphi_2)\in\mathbb{F}\) such that \(\sigma^u \in \mathscr{S}^u(\tau_1,\varphi_1)\) and \(\sigma_l \in \mathscr{S}_l(\tau_2,\varphi_2)\).
		There is a face of \(\tau_i\) that contains \(\pi(x)\) in its relative interior.
		This face is then a genuine colored cone, and it is straightfoward to show that it contains \(\sigma^u\) respectively \(\sigma_l\).
		Let us continue to use the notation \((\tau_1,\varphi_1)\) and \((\tau_2,\varphi_2)\) for these faces, so that we have \({(\sigma^u)}^\circ\subseteq\tau_1^\circ\) and \(\sigma_l^\circ\subseteq\tau_2^\circ\).
		Since \(\pi(x)\in\tau_1^\circ\cap\tau_2^\circ\cap\V\), we have \((\tau_1,\varphi_1)=(\tau_2,\varphi_2)\eqcolon(\tau,\varphi)\).
		Then by \Cref{prop:compatible-subdivision}, \(\sigma^u\in\mathscr{S}^u(\tau,\varphi)\) and \(\sigma_l\in\mathscr{S}_l(\tau,\varphi)\).
        
        So there is a linear functional \(\psi\), which is negative in the last coordinate, such that \(\psi((\sigma^u)')=0\), \(\psi(\sigma_l')\ge0\), and \(\psi(\rho)<0\).
		Therefore, \(\psi({((\sigma^u)'+\rho)}^\circ)=\psi((\sigma^u)'^\circ)+\psi(\rho^\circ)<0\), and \(\psi({(\sigma_l'-\rho)}^\circ)=\psi(\sigma_l'^\circ)-\psi(\rho^\circ)>0\).
		So \(\mathcal{C}_1^\circ\cap\mathcal{C}_2^\circ\cap\V'=\emptyset\).
    \end{enumerate} 
    Therefore \(\mathbb{F}_3\) is a colored fan.
\end{proof}

\begin{prop}\noindent\phantomsection\label{prop:F3-complete-fan}
    \(\mathbb{F}_3\) is a complete colored fan.
\end{prop}

\begin{proof}
    We show that \(\V'\subseteq|\mathbb{F}_3|\).
    Let \((x,\lambda) \in \V'\), then \(\pi(x,\lambda) = x \in \mathcal{C}\) for some \((\mathcal{C}, \mathcal{F})\in\mathbb{F}\).
	By convexity \((x\times\QQ) \cap \mathcal{C}' = [h^u,h_l]\) with \((x,h_l)\in\sigma'_l\) and \((x,h^u)\in(\sigma^u)'\) for some \(\sigma^u\in\mathscr{S}^u(\mathcal{C},\mathcal{F})\), \(\sigma_l\in\mathscr{S}_l(\mathcal{C},\mathcal{F})\).
    So \((x,\lambda)\) is contained in \(((\sigma^u)'+\rho)\cup|\mathbb{F}_2|\cup(\sigma'_l+(-\rho))\subseteq |\mathbb{F}_3|\).
\end{proof}

We are now prepared to complete the proof of \Cref{prop:Gorenstein algorithm}.

\begin{proof}[Proof of \Cref{prop:Gorenstein algorithm}]
	\Cref{Gorenstein algorithm} terminates as there are finitely many colors, and finitely many abstract colored cones at each stage.
	Since \(|\mathbb{F}_3|=|\mathbb{F}_4|\), by \Cref{prop:F3-complete-fan} and \Cref{prop:star-subdiv}, \(\mathbb{F}_4\) is complete.
	In Step~\eqref{item:final-step}, an abstract colored cone \(\sigma\) in the triangulation of \(\mathbb{F}_4\) is thrown away if and only if its relative interior does not intersect the valuation cone, if and only if \(\sigma\) is an abstract colored cone which is not a genuine colored cone.
	Therefore, the resulting abstract colored fan is a genuine colored fan \(\mathbb{F}' \coloneq\mathbb{F}_5\) which is complete.

	By \Cref{Gorenstein Luna datum}, each colored cone in \(\mathbb{F}'\) is of the form described in \Cref{Q-Gorenstein fan}, so we have constructed the colored fan \(\mathbb{F}'\) of a complete \(\QQ\)-Gorenstein spherical variety \(X'\).
\end{proof}

\section{\texorpdfstring{Toric varieties and \(\widetilde{\wp}\)}{Toric varieties and P}}\noindent\phantomsection\label{sec:multiplicity free}
Let us continue to use the notation and assumptions from above.
In particular, \(X\) denotes a complete \(G\)-spherical variety.
In this section we prove that \(\widetilde{\wp}(X)<1\) if and only if \(X\) is isomorphic to a toric variety, if and only if \(\widetilde{\wp}(X)=0\), which completes the proof of \Cref{main thm}.
Indeed, we show
\[
    \widetilde{\wp}(X) <1 \Rightarrow X\text{ isomorphic to a toric variety} \Rightarrow \widetilde{\wp}(X)=0 \Rightarrow \widetilde{\wp}(X)<1.
\]

We use a technique from~\cite{GH15}, where it was observed that \(\widetilde{\wp}(X)\) only depends on the spherical skeleton, as introduced in \Cref{sec:intro}.

\begin{rem}\noindent\phantomsection\label{toric skeleton}
	There is a natural notion of isomorphism of spherical skeletons (see~\cite[Section~5]{GH15}), and by design \(\widetilde{\wp}(X)=\widetilde{\wp}(\mathscr{R}_X)\) only depends on \(\mathscr{R}_X\) up to isomorphism.
    For completeness and to ensure that this work is self-contained, we briefly recall when two spherical skeletons \(\mathscr{R}_1 = (\Sigma_1, S_1^p, \D_1^a, \Gamma_1, \varsigma_1, \rho'_1)\) and \(\mathscr{R}_2 = (\Sigma_2, S_2^p, \D_2^a, \Gamma_2, \varsigma_2, \rho'_2)\) are considered to be \emph{isomorphic}, written \(\mathscr{R}_1 \cong \mathscr{R}_2\).
    Suppose the underlying root system of \(\mathscr{R}_1\) (resp.~\(\mathscr{R}_2\)) is \(R_1\) (resp.~\(R_2\)).
    We say \(\mathscr{R}_1\) and \(\mathscr{R}_2\) are isomorphic if there exists an isomorphism of root systems \(\varphi_R \colon R_1 \to R_2\) and two bijections \(\varphi_\D \colon \D_1^a \to \D_2^a\) and \(\varphi_\Gamma\colon \Gamma_1 \to \Gamma_2\) such that \(\varphi_R(\Sigma_1) = \Sigma_2\), \(\varphi_R(S_1^p) = S_2^p\), and \(\rho'_1(D) = \rho'_2(\varphi_\D(D))\circ\varphi_R\) on \(\Sigma_1\) for every \(D \in \D_1^a\) as well as \(\rho'_1(D) = \rho'_2(\varphi_\Gamma(D))\circ \varphi_R\) on \(\Sigma_1\) for every \(D \in \Gamma_1\).    
\end{rem}

\begin{rem}\noindent\phantomsection\label{rem:skeleton-equivalence}
	Remark that \(\widetilde{\wp}(\mathscr{R})\) only depends on the representative of spherical skeleton up to the following equivalence relation, introduced in ~\cite[Theorem~6.7]{GH15}.
	For \(\mathscr{R}=(\Sigma,S^p,\mathcal{D}^a,\Gamma)\), define \([\Gamma]\coloneq\{D\in\Gamma\mid\rho'(D)\neq0\in\Lambda^*\}\), and \([\mathscr{R}]\coloneq(\Sigma,S^p,\mathcal{D}^a,[\Gamma])\).
	The aforementioned equivalence relation is then \(\mathscr{R}_1\sim\mathscr{R}_2\) if and only if \([\mathscr{R}_1]\cong[\mathscr{R}_2]\).
\end{rem}

The following theorem is a key result that is required in the proof of \Cref{main thm}.
\begin{thm}[{\cite[Theorem~6.7]{GH15}}]\noindent\phantomsection\label{thm:toric-is-mult-free-space}
	A complete \(G\)-spherical variety \(X\) is isomorphic to a toric variety if and only if \(\mathscr{R}_X \cong \mathscr{R}_V\) for \(V\) a \emph{\(G\)-multiplicity free space}; i.e.~a rational \(G\)-representation with the structure of a \(G\)-spherical variety (see also \Cref{def:mult-free-space}).
\end{thm}

\begin{rem}\noindent\phantomsection\label{rem:toric-crit}
	With \Cref{rem:skeleton-equivalence} and \Cref{thm:toric-is-mult-free-space}, we have the following criterion which detects a toric variety.
	Let \(X\) be a spherical variety that is isomorphic to a toric variety.
	Then, by \Cref{thm:toric-is-mult-free-space}, \(\mathscr{R}_X\cong\mathscr{R}_V\) for \(V\) a multiplicity free space.
	By~\cite[Remark~6.10]{GH15}, under the equivalence relation in \Cref{rem:skeleton-equivalence}, if for some spherical variety \(Y\), \(\mathscr{R}_{Y}\sim\mathscr{R}_X\), then \(\mathscr{R}_Y\) is also the spherical skeleton of a multiplicity free space.
	
\end{rem}

By \Cref{thm:toric-is-mult-free-space}, to prove the equality \(\widetilde{\wp}(X)=0\) when \(X\) is isomorphic to a toric variety, we may prove the corresponding equality for multiplicity free spaces:

\begin{prop}\noindent\phantomsection\label{mult free equality}
	Let \(V\) be a multiplicity free space, then \(\widetilde{\wp}(V)=0\).
\end{prop}

\begin{rem}
	Simultaneously and independently to the authors of this paper, in~\cite{bravi2025generalizedmukaiconjecturespherical}, Bravi and Pezzini established that \(\widetilde{\wp}(X)\ge0\) in the special case that \(X\) has a reductive general isotropy group.
	Their approach is a case by case verification of the conjecture on the corresponding spherical skeletons, as done for the symmetric varieties in \cite{GH15}.
	In their work, they show the equality \(\widetilde{\wp}(V)=0\) for \(V\) a multiplicity free space (\Cref{mult free equality} in this paper) by an abstract argument on the spherical skeleton and weight monoid of a multiplicity free space.
	Their approach to prove \Cref{mult free equality} is distinct from the one presented here.
\end{rem}

By the following \Cref{mult free reduction}, it suffices to show this equality for a simpler class of multiplicity free spaces, namely the \emph{indecomposable saturated multiplicity free spaces up to geometric equivalence}.
This class of multiplicity free spaces admits an explicit classification in terms of the associated spherical combinatorial data (see~\cite{KnopMultiplicityFree} and~\cite{GagliardiMultiplicityFree}). 
Before proving \Cref{mult free reduction}, we make some preliminary observations.

\begin{defin}\noindent\phantomsection\label{def:mult-free-space}
	Recall that a \emph{multiplicity free space} is a rational representation \(\varrho\colon G \to \operatorname{GL}(V)\) of \(G\) into a finite-dimensional vector space \(V\) which is also a spherical \(G\)-variety with respect to the induced linear \(G\)-action.
	Let \(G_1\) and \(G_2\) be two connected reductive groups.
	Then two multiplicity free spaces \(\varrho_1\colon G_1 \to \operatorname{GL}(V_1)\) and \(\varrho_2\colon G_2 \to \operatorname{GL}(V_2)\) are said to be \emph{geometrically equivalent} if there is an isomorphism \(\varphi\colon V_1 \to V_2\) of vector spaces, inducing \(\operatorname{GL}(\varphi)\colon \operatorname{GL}(V_1) \to \operatorname{GL}(V_2)\) such that \(\operatorname{GL}(\varphi)(\varrho_1(G_1)) =\varrho_2(G_2)\).

	The multiplicity free space \(V\) is said to be \emph{decomposable} if it is geometrically equivalent to \(\varrho'\colon G_1 \times G_2\to\GL(V_1\oplus V_2)\), where \(G_i\) acts on \(V_i\).
	
	Suppose the multiplicity free space \(V = V_1 \oplus \cdots \oplus V_s\) has \(s\) irreducible summands \(V_i\).
	Then \(V\) is called \emph{saturated} if the dimension of the central torus of \(\varrho(G)\) is equal to \(s\). 
\end{defin}

\begin{rem}
	Notice that by Schur's lemma, a multiplicity free space \(V = V_1 \oplus \cdots \oplus V_s\) is saturated if and only if \((\CC^*)^s\) is contained in the centre of \(\varrho(G)\) and the \(i\)-th \(\CC^*\)-factor acts via scalar multiplication on \(V_i\).
	In particular, the two definitions \cite[Definition in Section~5]{KnopMultiplicityFree} and \cite[Definition~2.3]{Leahy98} are equivalent.

	Furthermore, observe that from the definition of being saturated it is clear that a non-saturated multiplicity free space can be made saturated by enlarging the central torus of \(G\) (ensuring that \(G\) maps onto the central torus of \(\GL(V)\)).
\end{rem}

A technical subtlety when working with multiplicity free spaces \(\varrho\colon G \to \GL(V)\) is that they naturally come equipped with two actions by connected reductive groups, namely the one by \(G\) and one by \(\varrho(G) \subseteq \GL(V)\).
It is an interesting question to study how the spherical skeletons of \(V\) considered as spherical \(G\)- and \(\varrho(G)\)-variety are related.
In general their spherical skeletons need not be isomorphic in the sense above (even when considering them up to equivalence) as their ``ambient root systems'' might not be comparable.
However, there is a natural notion of ``minimal ambient root system'' for a spherical skeleton which would allow us to consider the two spherical skeletons of \(V\) as isomorphic.
This study is rather technical and, as we only need to compare the \(\widetilde{\wp}\) functions of these two ``spherical actions'', the following simpler statement suffices for our purposes.

\begin{lemma}\noindent\phantomsection\label{lemma:normal-subgroup-p}
	Let \(X_{G/H}\hookleftarrow G/H\) be a \(G\)-spherical variety.
	Suppose \(K\trianglelefteq G\) is a closed normal subgroup with \(K\subseteq H\), then for \(\mathfrak{G}\coloneq G/K\) and \(\mathfrak{H}\coloneq H/K\), we have that \(X\hookleftarrow\mathfrak{G}/\mathfrak{H}\) is a \(\mathfrak{G}\)-spherical variety with \(\widetilde{\wp}(X_{G/H})=\widetilde{\wp}(X_{\mathfrak{G}/\mathfrak{H}})\).
\end{lemma}

\begin{proof}
	Recall the definition of \(\widetilde{\wp}(X)\):
	\[
		\widetilde{\wp}(X) = \dim X-\rk X-\sup \left\{ \sum_{D\in\Delta}\left(m_D-1+\big\langle\rho(D),\vartheta\big\rangle\right)\,\big|\,\vartheta\in Q^*\cap\T \right\},
	\]
	where \(Q^*\) is uniquely determined by the valuations induced by the \(B\)-invariant divisors, and the coefficients of those \(B\)-invariant divisors in the spherical anticanonical divisor.
	Therefore, in order to show \(\widetilde{\wp}(X_{G/H}) = \widetilde{\wp}(X_{\mathfrak{G}/\mathfrak{H}})\), it suffices to verify that the lattices of weights of semi-invariant rational functions with respect to a Borel subgroup can be identified in such a way that the valuation cones coincide, and the spherical anticanonical divisors are equal.
	In other words, we need to show that the Luna-Vust theory of \(G/H\) and \(\mathfrak{G}/\mathfrak{H}\) coincide, as well as \(-K_{G/H}=-K_{\mathfrak{G}/\mathfrak{H}}\).

	Let \(\pi \colon G \to \mathfrak{G} = G / K\) be the natural quotient, yielding an isomorphism \(\varphi \colon G/H \to \mathfrak{G} / \mathfrak{H}\).
	Fix a maximal torus and Borel subgroup \(T\subseteq B\subseteq G\), then \(\pi(T)\subseteq \mathfrak{B} \coloneqq \pi(B)\subseteq\mathfrak{G}\) are the corresponding maximal torus and Borel subgroup of \(\mathfrak{G}\).
	Let \(\M\coloneq {\CC(G/H)}^{(B)}/\CC^*\) be the character lattice of \(G/H\), and \(\mathfrak{M}\coloneq {\CC(\mathfrak{G/\mathfrak{H}})}^{(\mathfrak{B})}/\CC^*\) be the character lattice of \(\mathfrak{G}/\mathfrak{H}\).
	Then \(\varphi\) induces an isomorphism of character lattices \(\varphi^*\colon \mathfrak{M} \to \M\), and of dual lattices \(\varphi_* \colon \N \to \mathfrak{N}\).
	Under this isomorphism, the valuation cones of \(G/H\) and \(\mathfrak{G}/\mathfrak{H}\) coincide, i.e.\ \(\varphi_*(\V) = \mathfrak{V}\) (use the well-known identification of the valuation cone with \(G\)-invariant (resp.~\(\mathfrak{G}\)-invariant) valuations on \(\CC(G/H)\) (resp.~\(\CC(\mathfrak{G}/\mathfrak{H})\))).
	Finally, \(\varphi\) induces a bijection between the \(B\)-orbits of \(G/H\) and the \(\mathfrak{B}\)-orbits of \(\mathfrak{G}/\mathfrak{H}\) such that for every \(B\)-orbit \(Y\) of \(G/H\), the corresponding \(\mathfrak{B}\)-orbit is \(\varphi(Y)\) and \(\varphi\colon Y \to \varphi(Y)\) is a \(B\)-equivariant isomorphism.
	Therefore, the Luna-Vust theories of \(G/H\) and \(\mathfrak{G}/\mathfrak{H}\) coincide.
	In particular, their spherical embeddings can be naturally identified.

	Next we show that \(-K_{G/H}=-K_{\mathfrak{G}/\mathfrak{H}}\).
	Let \(U\subseteq G/H\) be the open \(B\)-orbit and let \(s\in\Gamma(U,\hat{\omega}_{G/H})\) be a generator such that \(\divv(s) = -K_{G/H}\).
	Notice that \(s\) is uniquely determined up to a constant multiple.
	Above, we have seen that \(\varphi\) induces an isomorphism between the open \(B\)-orbit \(U\) and the open \(\mathfrak{B}\)-orbit \(\varphi(U) \subseteq \mathfrak{G}/\mathfrak{H}\).
	Under this isomorphism \(\hat{\omega}_{G/H}\) gets identified with \(\hat{\omega}_{\mathfrak{G}/\mathfrak{H}}\) and the generator \(s\) corresponds to a generator \(\mathfrak{s} \in \Gamma(\varphi(U), \hat{\omega}_{\mathfrak{G}/\mathfrak{H}})\).
	By \cite[Theorem~1.4]{GorensteinFano}, the generator \(s\) is uniquely determined by the property that it vanishes with order one along each \(G\)-invariant divisor in every spherical embedding \(G/H\hookrightarrow X\).
	As the Luna-Vust data of \(G/H\) and \(\mathfrak{G}/\mathfrak{H}\) coincide, the generator \(\mathfrak{s}\) has the same property, and thus \(-K_{\mathfrak{G}/\mathfrak{H}} = \divv(\mathfrak{s}) = \divv(s) = -K_{G/H}\).
\end{proof}

\begin{cor}\noindent\phantomsection\label{cor:p-fct-mult-free-space}
	Let \(\varrho\colon G \to \GL(V)\) be a multiplicity free space.
	Then the \(\widetilde{\wp}\)-functions of \(V\) considered as a \(G\)- or \(\varrho(G)\)- spherical variety coincide, i.e.\ \(\widetilde{\wp}(V_G) = \widetilde{\wp}(V_{\varrho(G)})\).
\end{cor}
\begin{proof}
	This straightforwardly follows from \Cref{lemma:normal-subgroup-p} applied to the \(G\)-spherical variety \(V\) and \(K = \ker{\varrho}\).
\end{proof}

\begin{cor}\noindent\phantomsection\label{cor:geometric-equivalence-preserves-p}
	Let \(\varrho_1\colon G_1 \to \operatorname{GL}(V_1)\) and \(\varrho_2\colon G_2 \to \operatorname{GL}(V_2)\) be geometrically equivalent multiplicity free spaces.
	Then \(\widetilde{\wp}(V_1)=\widetilde{\wp}(V_2)\).
\end{cor}

\begin{proof}
	By \Cref{cor:p-fct-mult-free-space}, we may consider \(V_i\) as \(\varrho(G_i)\)-spherical variety.
	Notice the assumption that \(V_1\) and \(V_2\) are geometrically equivalent implies there is an isomorphism \(\varrho(G_1) \cong \varrho(G_2)\) of algebraic groups and a linear isomorphism \(V_1 \cong V_2\) which is equivariant with respect to the identification \(\varrho(G_1) \cong \varrho(G_2)\).
	Now the statement directly follows.
\end{proof}

\begin{prop}\noindent\phantomsection\label{mult free reduction}
	Suppose that for each indecomposable saturated multiplicity free space \(V'\), we have \(\widetilde{\wp}(V')=0\).
	Then \(\widetilde{\wp}(V)=0\) for every multiplicity free space \(V\).
\end{prop}

\begin{proof}
	Let \(V=V_1 \oplus \dots \oplus V_s\) be a \(G\)-multiplicity free space with \(s\) irreducible summands \(V_i\).
	By Schur's lemma, the dimension of the central torus of \(\varrho(G)\) is not greater than \(s\).
	Then by replacing \(G\) with \(G' = G \times {(\CC^*)}^k\), for a suitable \(k\ge0\), we may obtain a saturated \(G'\)-multiplicity free space \(\widetilde{V}\) (the same linear space \(V\) with an ``enlarged action'').
	It is clear that adding torus factors preserves the spherical skeleton, i.e.\ \(\mathscr{R}_{\widetilde{V}}=\mathscr{R}_{V}\), so in particular \(\widetilde{\wp}(\widetilde{V})=\widetilde{\wp}(V)\).
	So we may assume that \(V\) is saturated.

	Now let us suppose that \(V\) is decomposable.
	Then by \Cref{cor:geometric-equivalence-preserves-p}, we may assume that \(G=G_1\times G_2\) and \(V=V_1\oplus V_2\), with \(G_i\) acting on \(V_i\).
	Notice that from a spherical geometry point of view \(V\) is a product of two spherical varieties \(G_1/H_1 \hookrightarrow V_1\) and \(G_2/H_2 \hookrightarrow V_2\).
	Therefore, the Luna datum of \(V\) decomposes into a direct product of the two Luna data of \(V_1\) and \(V_2\) (see, for instance, \cite[Appendix]{Bravi13}.
	It is now straightforward to verify that
	\[
		\widetilde{\wp}(V)=\widetilde{\wp}(V_1)+\widetilde{\wp}(V_2).
	\] 
	So by \Cref{cor:geometric-equivalence-preserves-p}, showing \(\widetilde{\wp}(V')=0\) for every saturated multiplicity free space up to geometric equivalence, implies that \(\widetilde{\wp}(V)=0\).
\end{proof}

\begin{proof}[Proof of \Cref{mult free equality}]
    In \Cref{Table}, the equality \(\widetilde{\wp}(V)=0\) is established for the indecomposable saturated multiplicity free spaces \(V\) up to geometric equivalence, by explicitly solving the linear program \(\widetilde{\wp}(V)\) in each of the cases in~\cite{KnopMultiplicityFree}.
    By \Cref{mult free reduction}, this establishes the equality in general.
\end{proof}

This completes the proof of \Cref{main thm}.

\begin{proof}[Proof of \Cref{main thm}]
    By \Cref{complete P theorem}, it remains to show:
    \[
		\widetilde{\wp}(X) <1 \Rightarrow X\text{ isomorphic to a toric variety} \Rightarrow \widetilde{\wp}(X)=0 \Rightarrow \widetilde{\wp}(X)<1.
    \]

	The strategy of the proof is to apply \Cref{Gorenstein algorithm} to \(X\) to obtain a complete \(\QQ\)-Gorenstein spherical variety \(X'\). Then we use \Cref{Spherical Mukai}. The proof follows by studying how \Cref{Gorenstein algorithm} interacts with the spherical skeletons of \(X\) and \(X'\).

    Firstly, suppose \(\widetilde{\wp}(X)<1\).
	Then by \Cref{Spherical Mukai}, \(X'\) is isomorphic to a toric variety, and thus, by \Cref{thm:toric-is-mult-free-space}, \(\mathscr{R}_{X'} \cong \mathscr{R}_{V'}\) for a multiplicity free space \(V'\).
	We need to prove that \(X\) is isomorphic to a toric variety.
    To do this, we study how \Cref{Gorenstein algorithm} interacts with the spherical skeleton of \(X\).
	It is clear that the modifications in \Cref{Gorenstein algorithm} do not change the equivalence class (see \Cref{rem:skeleton-equivalence}) of the spherical skeleton \(\mathscr{R}_X\sim\mathscr{R}_{X'}\).
	Therefore, by \Cref{rem:toric-crit}, \(\mathscr{R}_X\cong \mathscr{R}_V\) for a multiplicity free space \(V\).
	Hence, by \Cref{thm:toric-is-mult-free-space}, \(X\) is isomorphic to a toric variety.

    Secondly, we need to prove that \(\widetilde{\wp}(X)=0\) if \(X\) is isomorphic to a toric variety.
	By~\Cref{thm:toric-is-mult-free-space}, \(X\) being isomorphic to a toric variety is equivalent to \(\mathscr{R}_X \cong \mathscr{R}_V\) for a multiplicity free space \(V\).
	Hence, \(\widetilde{\wp}(X) = \widetilde{\wp}(V)=0\) by \Cref{mult free equality}.
\end{proof}

%--------------------------------------------------------------------------------
\section{Examples}\noindent\phantomsection\label{sec:examples}
In this section we illustrate \Cref{smoothness theorem}.
We exhibit an example of a singular locally factorial simple embedding. This is a non-toric phenomenon, as for toric varieties smoothness coincides with local factoriality. 

\begin{example} Consider the variety of smooth conics \(G/H=\mathrm{SL}_3/Z(\mathrm{SL}_3)\mathrm{SO}_3\) (cf.~\cite[Table~A]{Wasserman}).
	We have \(\M=\ZZ(2\alpha_1)\oplus\ZZ(2\alpha_2)\), two colors \(D_1\), \(D_2\) of type 2a with spherical roots \(\Sigma=\{2\alpha_1, 2\alpha_2\}\), and dual vectors \(\rho(D_1)=(2,-1)\), \(\rho(D_2)=(-1,2)\in\N\).
	The simple embedding \(X\) with colored cone \((\mathrm{cone}((-1,0),(2,-1)),\{D_1\})\) is the parameter space \(\PP^5\) of all conics, with one spherical boundary divisor \(X_1\), and \(\Delta=\{X_1,D_1,D_2\}\).
	The polytope \(Q^*\) is the triangle \(\mathrm{conv}((-1,-1),(1,0),(1,3))\), and the tailcone \(\mathcal{T}=\mathrm{cone}(\Sigma)\) is the positive orthant of \(\M_\QQ\).
        As a demonstration, we apply \Cref{smoothness theorem} to the unique closed \(G\)-orbit \(Y\) of \(\PP^5\hookleftarrow G/H\). 
        The spherical skeleton of \(\PP^5\hookleftarrow G/H\) is
        \[
            \mathscr{R}=(\Sigma=\{2\alpha_1,2\alpha_2\},S^p=\emptyset,\D^a=\emptyset,\Gamma=\{X_1\}).
        \] 
        The \(B\)-invariant prime divisors \(D\) with \(Y\subseteq \overline{D}\) are \(D\in I=\{X_1,D_1\}\), the new simple root is \(S_I=\{\alpha_1\}\), and the localisation of \(\mathscr{R}\) at \(I\) is \[\mathscr{R}_I=(\Sigma_I=\{2\alpha_1\},S^p_I=\emptyset,\D^a=\emptyset,\Gamma_I=\{X_1\}).\] 
        
		Then \(\rho_I'(D_1)=2\in\Lambda^*_I\), and \(\rho_I'(X_1)=-1\in\Lambda^*_I\). Since \(D\) is not of type b, we have \((m_D)_I=(m_{X_1})_I=1\), so \(\mathcal{Q}^*_{\mathscr{R}_I}\cap\cone(\Sigma_I)=[0,1]\subset {(\Lambda_I)}_\QQ\).
        Finally, \(\widetilde{\wp}(\mathscr{R}_I)=0\), so by \Cref{smoothness theorem}, \(\PP^5\) is smooth along \(Y\) as expected. 
\end{example}

\begin{example}[A singular locally factorial variety]
    Let \(X\) be a simple spherical embedding with colored cone \((\mathcal{C},\mathcal{F})\), with no colors taken \(\mathcal{F}=\emptyset\), and \(\mathcal{C}\) a smooth cone.
    Then by \Cref{smoothness theorem}, \(X\) is smooth.
    However, when \(\mathcal{F}\neq\emptyset\), \(X\) can be singular.
    
    For example, consider the spherical homogeneous space \(G/H=(\SL_2\times \CC^*)/T'\times\SL_2/T\), with \(T'\) the diagonal torus of \(\SL_2\), acting with character \(\omega\) on \(\CC^*\), and \(T\) the diagonal torus of \(\SL_2\). It has \(\N_\QQ\cong\QQ^3\), \((\rho(D)\mid D\in\D)=((1,0,0),(0,1,0),(0,0,1),(0,0,1))\), and its valuation cone is \(\V=\cone((1,-1,0),(-1,1,0),(-1,-1,0),(0,0,-1))\).
    Let \(G/H\hookrightarrow X\) be the simple spherical embedding whose colored cone has rays \((-1,-1,-1),(1,0,0),(0,1,0)\). This is a smooth cone, with one color taken, so \(X\) is locally factorial. We compute
    \(|R^+|-|R_{S^p}^+|=2\), \(\sum_{D\in\Delta}(m_D-1)=0\), and \(\sum_{D\in\Delta}\rho(D)=(0,0,1)\). From this, we compute \(\widetilde{\wp}(X)=1\), therefore by \Cref{smoothness theorem}, \(X\) is locally factorial but not smooth, and such a situation does not occur in toric geometry. 
\end{example}

%-------------------------------------------------------------------------------------

\section{\texorpdfstring{\(\widetilde{\wp}\) for the indecomposable saturated multiplicity free spaces}{p for the indecomposable saturated multiplicity free spaces.}}\noindent\phantomsection\label{Table}
In this section, we show that \(\widetilde{\wp}(V)=0\) for \(V\) an indecomposable saturated multiplicity free space up to geometric equivalence, using Knop's exhaustive list~\cite[Section~5]{KnopMultiplicityFree} of these spherical varieties.
In each case, we find an explicit value \(\vartheta\in Q^*\cap\T\) which evaluates in the linear program \(\widetilde{\wp}(V)\) to zero.
Then the following \Cref{appendix prop} implies that \(\widetilde{\wp}(V)=0\).

\begin{prop}\noindent\phantomsection\label{appendix prop}
	Let \(V\) be a multiplicity free space, then \(\widetilde{\wp}(V)\ge0\).
\end{prop}

\begin{proof}
	By~\cite[Lemma~7.3]{GH15}, the spherical skeleton of a multiplicity free space is isomorphic to the spherical skeleton of a complete spherical variety.
	The inequality now follows from \Cref{complete P theorem}.
\end{proof}

By \Cref{appendix prop}, this explicit value \(\vartheta_{\mathrm{argmax}}\coloneq\vartheta\) solves the linear program \(\wp(V)=\dim V-\rk V-\widetilde{\wp}(V)\) and establishes the equality \(\widetilde{\wp}(V)=0\).
In the following table, the value of \(\vartheta_{\mathrm{argmax}}\) is written in terms of the spherical roots for each of the multiplicity free spaces in Knop's list.
In each case, we denote by \(\lambda_i\), the \(i\)th spherical root following the enumeration in~\cite{KnopMultiplicityFree}.
In~\cite{GagliardiMultiplicityFree}, Gagliardi has computed the Luna diagram of each of these multiplicity free spaces, and throughout, we follow the numbering of cases used in that paper. 

Let \(V\) be a \(G\)-multiplicity free space, then to compute \(\vartheta_{\mathrm{argmax}}\) we use the following.
Since \(V\) is an affine spherical variety, it is a simple spherical embedding with colored cone \((\mathcal{C},\mathcal{F})\) by~\cite[Theorem~6.7]{Knop2012}.
Since the origin is a \(G\)-fixed point of \(V\), by the orbit-cone correspondence~\cite[Theorem~6.6]{Knop2012}, \((\mathcal{C},\mathcal{F})\) is a colored cone of maximal dimension.
Moreover, as \(V\) is factorial, the primitive ray generators of rays \(\{\rho(D)\,|\,D\in\Delta\}\) of its colored cone form a \(\ZZ\)-basis of \(\N\).
Furthermore, the \(B\)-weights \(\chi\) of each \(B\)-invariant principal divisor \(D=\divv(f_\chi)\) form the corresponding dual basis \(\{\chi_D\,|\,D\in\Delta\}\) of \(\M\).
In this basis, the objective function of the linear program \(\wp(V)\) is the sum of the components of a vector, and \(Q^*=(-m_{D}\,|\,D\in\Delta)+{(\M_\QQ)}_{\ge0}\) is a translation of the positive orthant.
Here, \((-m_D \mid D \in \Delta)\) denotes the point in \(\M\) which has ``\(D\)-th'' coordinate equal to \(-m_D\).
Writing each spherical root \(\lambda_i\) in terms of the \(\chi_D\), in each case we find a value \(\vartheta_{\mathrm{argmax}}=\sum_i a_i\lambda_i\) such that the coefficients \(a_i\) satisfy a linear recurrence.

%--------------------------------------------------------------------------------

\subsection{Explicit values of \texorpdfstring{\(\vartheta_{\mathrm{argmax}}\)}{vartheta\_argmax}}
The following list displays the value \(\vartheta_{\mathrm{argmax}}\) which solves the linear program \(\widetilde{\wp}(X)\) for each of the indecomposable saturated multiplicity free spaces \(X\) up to geometric equivalence.

For convenience, denote by \(\delta_k\) the value of \(k\) modulo \(2\). 
\begin{enumerate}[label=(\arabic*)]
\setlength{\itemsep}{1em}

\item
    \(\SL_n\times\CC^*\) on \(\CC^n\) with \(n\ge2\),

    \nopagebreak\(\vartheta_{\mathrm{argmax}}=0\).

\item
    \(\Sp_{2n}\times\CC^*\) on \(\CC^{2n}\) with \(n\ge2\),

    \nopagebreak\(\vartheta_{\mathrm{argmax}}=0\).

\item
    \(\mathrm{Spin}_{2n+1}\times\CC^*\) on \(\CC^{2n+1}\) with \(n\ge2\),

    \nopagebreak\(\vartheta_{\mathrm{argmax}}=\lambda_1\).

\item
    \(\mathrm{Spin}_{2n}\times\CC^*\) on \(\CC^{2n}\) with \(n\ge3\),

    \nopagebreak\(\vartheta_{\mathrm{argmax}}=\lambda_1\).

\item
    \(\SL_n\times\CC^*\) on \(S^2(\CC^n)\) with \(n\ge2\),

    \nopagebreak\(\vartheta_{\mathrm{argmax}}=\sum_{k=1}^{n-1}\frac{1}{2}k(k+1)\lambda_{n-k}\).

\item
    \(\SL_n\times\CC^*\) on \(\Lambda^2(\CC^n)\) with \(n\ge5\) odd,

    \nopagebreak\(\vartheta_{\mathrm{argmax}}=\sum_{k=1}^{\lfloor n/2\rfloor-1}k(2k+1)\lambda_{\lfloor n/2\rfloor-k}\).

\item
    \(\SL_n\times\CC^*\) on \(\Lambda^2(\CC^n)\) with \(n\ge6\) even,

    \nopagebreak\(\vartheta_{\mathrm{argmax}}=\sum_{k=1}^{n/2-1}k(2k-1)\lambda_{n/2-k}\).

\item
    \(\SL_n\times\SL_n\times\CC^*\) on \(\CC^n\otimes\CC^n\) with \(n\ge2\),

    \nopagebreak\(\vartheta_{\mathrm{argmax}}=\sum_{k=1}^{n-1}k^2\lambda_{n-k}\).

\item
    \(\SL_n\times\SL_{n'}\times\CC^*\) on \(\CC^n\otimes\CC^{n'}\) with \(n>n'\ge2\),

    \nopagebreak\(\vartheta_{\mathrm{argmax}}=\sum_{k=1}^{n-1}k(k+n-n')\lambda_{n-k}\).

\item
    \(\SL_2\times\Sp_{2n'}\times\CC^*\) on \(\CC^2\otimes\CC^{2n'}\) with \(n'\ge2\),

    \nopagebreak\(\vartheta_{\mathrm{argmax}}=(2n'-1)\lambda_1+\lambda_2\).

\item
    \(\SL_3\times\Sp_{4}\times\CC^*\) on \(\CC^3\otimes\CC^{4}\) with \(n'\ge2\),

    \nopagebreak\(\vartheta_{\mathrm{argmax}}=8\lambda_1+3\lambda_2+6\lambda_3+2\lambda_4\).

\item
    \(\SL_3\times\Sp_{2n}\times\CC^*\) on \(\CC^3\otimes\CC^{2n'}\) with \(n'\ge3\),

    \nopagebreak\(\vartheta_{\mathrm{argmax}}=4n\lambda_1+(2n-1)\lambda_2+3\lambda_3+(4n-2)\lambda_4+(2n-2)\lambda_5\).

\item
    \(\SL_4\times\Sp_4\times\CC^*\) on \(\CC^4\otimes\CC^4\),

    \nopagebreak\(\vartheta_{\mathrm{argmax}}=12\lambda_1+5\lambda_2+9\lambda_3+4\lambda_4+\lambda_5\).

\item
    \(\SL_n\times\Sp_4\times\CC^*\) on \(\CC^n\otimes\CC^4\) with \(n\ge5\),

    \nopagebreak\(\vartheta_{\mathrm{argmax}}=(4n-4)\lambda_1+(2n-3)\lambda_2+(3n-3)\lambda_3+(2n-4)\lambda_4+(n-3)\lambda_5\).

\item
    \(\mathrm{Spin}_7\times\CC^*\) on \(\CC^8\),

    \nopagebreak\(\vartheta_{\mathrm{argmax}}=\lambda_1\).

\item
    \(\mathrm{Spin}_9\times\CC^*\) on \(\CC^{16}\),

    \nopagebreak\(\vartheta_{\mathrm{argmax}}=\lambda_1+5\lambda_2\).

\item
    \(\mathrm{Spin}_{10}\times\CC^*\) on \(\CC^{16}\),

    \nopagebreak\(\vartheta_{\mathrm{argmax}}=5\lambda_1\).

\item
    \(G_2\times\CC^*\) on \(\CC^7\),

    \nopagebreak\(\vartheta_{\mathrm{argmax}}=\lambda_1\).

\item
    \(E_6\times\CC^*\) on \(\CC^{27}\),

    \nopagebreak\(\vartheta_{\mathrm{argmax}}=10\lambda_1+\lambda_2\).

\item
    \(\mathrm{Spin}_8\times\CC^*\) on \(\CC^8\oplus\CC^8\),

    \nopagebreak\(\vartheta_{\mathrm{argmax}}=\lambda_1+\lambda_2\).

\item
	\(\SL_2\times{(\CC^*)}^2\) on \(\CC^2\oplus\CC^2\),

    \nopagebreak\(\vartheta_{\mathrm{argmax}}=\lambda_1\).

\item
	\(\SL_n\times{(\CC^*)}^2\) on \(\CC^n\oplus\CC^n\) with \(n\ge3\),

    \nopagebreak\(\vartheta_{\mathrm{argmax}}=(n-1)\lambda_1\).

\item
	\(\SL_n\times{(\CC^*)}^2\) on \(\CC^n\oplus{(\CC^n)}^*\) with \(n\ge3\),

    \nopagebreak\(\vartheta_{\mathrm{argmax}}=\lambda_1\).

\item
	\(\SL_n\times{(\CC^*)}^2\) on \(\CC^n\oplus\Lambda(\CC^n)\) with \(n\ge3\),

    \nopagebreak\(\vartheta_{\mathrm{argmax}}=\sum_{k=1}^{n-2}\frac{1}{2}k(k+1)\lambda_{n-1-k}\).

\item
	\(\SL_n\times{(\CC^*)}^2\) on \({(\CC^n)}^*\oplus\Lambda(\CC^n)\) with \(n\ge4\) even,

    \nopagebreak\(\vartheta_{\mathrm{argmax}}=\sum_{k=1}^{n-2}(\frac{1}{2}k(k-1)+(n-1)\delta_k)\lambda_{n-1-k}\).

\item
	\(\SL_n\times{(\CC^*)}^2\) on \({(\CC^n)}^*\oplus\Lambda(\CC^n)\) with \(n\ge5\) odd,

    \nopagebreak\(\vartheta_{\mathrm{argmax}}=(n-1)\lambda_{n-2}+\sum_{k=1}^{n-3}(\frac{1}{2}k(k+1)+(n-1)\delta_k)\lambda_{n-2-k}\).

\item
	\(\SL_n\times\SL_{n'}\times{(\CC^*)}^2\) on \((\CC^n\otimes\CC^{n'})\oplus\CC^{n'}\) with \(2\le n<n'-1\),

    \nopagebreak\(\vartheta_{\mathrm{argmax}}=\sum_{k=1}^{n-1}k(n'-n+k)\lambda_k+\sum_{k=1}^{n'-1}k(n'-n+k-1)\lambda_{n-1+k}\).

\item
	\(\SL_n\times\SL_{n'}\times{(\CC^*)}^2\) on \((\CC^n\otimes\CC^{n'})\oplus\CC^{n'}\) with \(2\le n=n'-1\),

    \nopagebreak\(\vartheta_{\mathrm{argmax}}=\sum_{k=1}^{n-1}k(k+1)\lambda_k+\sum_{k=1}^{n}k^2\lambda_{n-1+k}\).

\item
	\(\SL_n\times\SL_{n'}\times{(\CC^*)}^2\) on \((\CC^n\otimes\CC^{n'})\oplus\CC^{n'}\) with \(2\le n=n'\),

    \nopagebreak\(\vartheta_{\mathrm{argmax}}=\sum_{k=1}^{n-1}k^2\lambda_k+\sum_{k=1}^{n-1}k(k+1)\lambda_{n-1+k}\).

\item
	\(\SL_n\times\SL_{n'}\times{(\CC^*)}^2\) on \((\CC^n\otimes\CC^{n'})\oplus\CC^{n'}\) with \(2\le n'<n\),

    \nopagebreak\(\vartheta_{\mathrm{argmax}}=\sum_{k=1}^{n'-1}(k(n-n'+k)-n+n')\lambda_k+\sum_{k=1}^{n'-1}k(n-n'+k+1)\lambda_{n'-1+k}\).

\item
	\(\SL_n\times\SL_{n'}\times{(\CC^*)}^2\) on \((\CC^n\otimes\CC^{n'})\oplus{(\CC^{n'})}^*\) with \(2\le n< n'-1\),

    \nopagebreak\(\vartheta_{\mathrm{argmax}}=\sum_{k=1}^{n-1}k(k+n'-n)\lambda_k+\sum_{k=1}^{n}(k(k+n'-n-3)+(2n-n'+2))\lambda_{n-1+k}\).

\item
	\(\SL_n\times\SL_{n'}\times{(\CC^*)}^2\) on \((\CC^n\otimes\CC^{n'})\oplus{(\CC^{n'})}^*\) with \(2\le n=n'-1\),

	\nopagebreak\(\vartheta_{\mathrm{argmax}}=\sum_{k=1}^{n-1}k(k+1)\lambda_k+\sum_{k=1}^n(n+{(k-1)}^2)\lambda_{n-1+k}\).

\item
	\(\SL_n\times\SL_{n'}\times{(\CC^*)}^2\) on \((\CC^n\otimes\CC^{n'})\oplus{(\CC^{n'})}^*\) with \(2\le n=n'\),

    \nopagebreak\(\vartheta_{\mathrm{argmax}}=\sum_{k=1}^{n-1}k^2\lambda_k+\sum_{k=1}^{n-1}(n+k(k-1))\lambda_{n-1+k}\).

\item
	\(\SL_n\times\SL_{n'}\times{(\CC^*)}^2\) on \((\CC^n\otimes\CC^{n'})\oplus{(\CC^{n'})}^*\) with \(2\le n'<n\),

    \nopagebreak\(\vartheta_{\mathrm{argmax}}=\sum_{k=1}^{n'-1}k(n-n'+k)\lambda_k+\sum_{k=1}^{n'-1}(k(n-n'+k-1)-n')\lambda_{n'-1+k}\).

\item
	\(\SL_2\times\SL_2\times \SL_2\times{(\CC^*)}^2\) on \((\CC^2\otimes\CC^2)\oplus(\CC^2\otimes\CC^2)\),

    \nopagebreak\(\vartheta_{\mathrm{argmax}}=\lambda_1+3\lambda_2+\lambda_3\).

\item
	\(\SL_n\times\SL_2\times \SL_2\times{(\CC^*)}^2\) on \((\CC^n\otimes\CC^2)\oplus(\CC^2\otimes\CC^2)\) with \(n\ge3\),

    \nopagebreak\(\vartheta_{\mathrm{argmax}}=(n-1)\lambda_1+(n+1)\lambda_2+\lambda_3\).

\item
	\(\SL_n\times\SL_2\times\SL_{n''}\times{(\CC^*)}^2\) on \((\CC^n\otimes\CC^2)\oplus(\CC^2\otimes\CC^{n''})\) with \(n\ge n''\ge3\),

    \nopagebreak\(\vartheta_{\mathrm{argmax}}=(n-1)\lambda_1+(n+n''-1)\lambda_2+(n''-1)\lambda_3\).

\item
	\(\Sp_{2n}\times{(\CC^*)}^2\) on \(\CC^{2n}\oplus\CC^{2n}\) with \(n\ge 2\),

    \nopagebreak\(\vartheta_{\mathrm{argmax}}=(2n-1)\lambda_1+\lambda_2\).

\item
	\(\Sp_{2n}\times\SL_2\times{(\CC^*)}^2\) on \((\CC^{2n}\otimes\CC^{2})\oplus\CC^2\) with \(n\ge 2\),

    \nopagebreak\(\vartheta_{\mathrm{argmax}}=(2n-1)\lambda_1+\lambda_2+(2n+2)\lambda_3\).

\item
	\(\Sp_{2n}\times\SL_2\times\SL_2\times{(\CC^*)}^2\) on \((\CC^{2n}\otimes\CC^2)\oplus(\CC^2\otimes\CC^2)\) with \(n\ge 2\),

    \nopagebreak\(\vartheta_{\mathrm{argmax}}=(2n-1)\lambda_1+\lambda_2+(2n+1)\lambda_3+\lambda_4\).

\item
	\(\Sp_{2n}\times\SL_2\times\SL_{n''}\times{(\CC^*)}^2\) on \((\CC^{2n}\otimes\CC^2)\oplus(\CC^2\otimes\CC^{n''})\) with \(n\ge 2,\,n''\ge3\),

    \nopagebreak\(\vartheta_{\mathrm{argmax}}=(2n-1)\lambda_1+\lambda_2+(2n+n''-1)\lambda_3+(n''-1)\lambda_4\).

\item
	\(\Sp_{2n}\times\SL_2\times\Sp_{2n''}\times{(\CC^*)}^2\) on \((\CC^{2n}\otimes\CC^2)\oplus(\CC^2\otimes\CC^{2n''})\) with \(n\ge 2,\,n''\ge3\),

    \nopagebreak\(\vartheta_{\mathrm{argmax}}=(2n-1)\lambda_1+\lambda_2+(2n+2n''-1)\lambda_3+(2n''-1)\lambda_4+\lambda_5\).

\end{enumerate}

%--------------------------------------------------------------------------------

\printbibliography{}

\end{document}